\documentclass{amsart}

\usepackage[T1]{fontenc}

\usepackage{amsmath,amsthm,amsfonts,amssymb,latexsym, mathrsfs,soul}

\makeatletter

\theoremstyle{plain}
\newtheorem{thm}{Theorem}[section]

\numberwithin{equation}{section} 

\numberwithin{figure}{section} 

\theoremstyle{plain}
\newtheorem{cor}[thm]{Corollary} 

\theoremstyle{definition}
\newtheorem{defn}[thm]{Definition}

\theoremstyle{plain}
\newtheorem{lem}[thm]{Lemma} 

\theoremstyle{plain}
\newtheorem{prop}[thm]{Proposition} 

\theoremstyle{plain}
\newtheorem{fact}[thm]{Fact}

\theoremstyle{plain}

\theoremstyle{plain}

\theoremstyle{definition}
\newtheorem{rem}[thm]{Remark}

\newtheorem{ques}[thm]{Question}

\newtheorem{theo}[thm]{Theorem}

\theoremstyle{plain}

\ifx\pdfoutput\@undefined\usepackage[usenames,dvips]{color}

\else\usepackage[usenames,dvipsnames]{color}

\IfFileExists{pdfcolmk.sty}{\usepackage{pdfcolmk}}{}

\fi

\makeatletter

\makeatletter
\def\moverlay{\mathpalette\mov@rlay}
\def\mov@rlay#1#2{\leavevmode\vtop{%
   \baselineskip\z@skip \lineskiplimit-\maxdimen
   \ialign{\hfil$\m@th#1##$\hfil\cr#2\crcr}}}
\newcommand{\charfusion}[3][\mathord]{
    #1{\ifx#1\mathop\vphantom{#2}\fi
        \mathpalette\mov@rlay{#2\cr#3}
      }
    \ifx#1\mathop\expandafter\displaylimits\fi}
\makeatother

\usepackage{amssymb}

\usepackage{amsmath}

\def \<{\langle}
\def \>{\rangle}
\def \pert{\operatorname{pert}}

\usepackage{amsfonts}

\usepackage{amscd}

\usepackage{latexsym}

\usepackage{epsfig}

\makeatother

\def\Ind#1#2{#1\setbox0=\hbox{$#1x$}\kern\wd0\hbox to 0pt{\hss$#1\mid$\hss}
\lower.9\ht0\hbox to 0pt{\hss$#1\smile$\hss}\kern\wd0}
\def\ind{\mathop{\mathpalette\Ind{}}}
\def\Notind#1#2{#1\setbox0=\hbox{$#1x$}\kern\wd0\hbox to 0pt{\mathchardef
\nn=12854\hss$#1\nn$\kern1.4\wd0\hss}\hbox to
0pt{\hss$#1\mid$\hss}\lower.9\ht0 \hbox to
0pt{\hss$#1\smile$\hss}\kern\wd0}
\def\nind{\mathop{\mathpalette\Notind{}}}

\def \fin{\operatorname{fin}}

\def \isol{\operatorname{Isol}}

\newcommand{\dminus}{ 
\buildrel\textstyle\ .\over{\hbox{ 
\vrule height3pt depth0pt width0pt}{\smash-} 
}}

\newcommand{\acl}{\operatorname{acl}}

\newcommand{\tp}{\operatorname{tp}}

\newcommand{\dcl}{\operatorname{dcl}}

\begin{document}

\title{Model theory of Hilbert spaces expanded by normal operators}

\author{Alexander Berenstein, Nicol\'as Cuervo Ovalle, and Isaac Goldbring}

\address{Department of Mathematics\\ Los Andes University, Cra. 1 \#18a-12, edificio H, Bogot\'a, Colombia, 111711.}
\urladdr{https://pentagono.uniandes.edu.co/~aberenst/index.html}
\thanks{The first-named author would like to thank the UC Irvine Department of Mathematics for their hospitality during a short visit in Spring 2025. He would also like to thank Universidad de los Andes for support during his sabbatical in Spring 2025}

\address{Department of Mathematics\\ Los Andes University, Cra. 1 \#18a-12, edificio H, Bogot\'a, Colombia, 111711.}
\email{n.cuervo10@uniandes.edu.co}
\thanks{The second-named author was partially supported by NSF grant DMS-2054477. He would also like to thank the UC Irvine Department of Mathematics for their hospitality.}

\address{Department of Mathematics\\University of California, Irvine, 340 Rowland Hall (Bldg.\# 400),
Irvine, CA 92697-3875}
\email{isaac@math.uci.edu}
\urladdr{http://www.math.uci.edu/~isaac}
\thanks{The third-named author was partially supported by NSF grant DMS-2054477.}

 \date{\today}

\maketitle

\begin{abstract}
 We study expansions of Hilbert spaces with a bounded normal operator $T$.  We axiomatize this theory in a natural language and identify all of its completions. We prove the definability of the adjoint $T^*$ and prove quantifier elimination for every completion after adding $T^*$ to the language. We identify types with measures on the spectrum of the operator and show that the logic topology on the type space corresponds to the weak*-topology on the space of measures.  We also give a precise formula for the metric on the space of $1$-types.  We prove all completions are stable and characterize the stability spectrum of the theory in terms of the spectrum of the operator. We also show all completions, regardless of their spectrum, are $\omega$-stable up to perturbations.   
\end{abstract}

\section{Introduction}
This paper deals with the model theory of expansions of Hilbert spaces by a bounded normal operator. Many basic facts about these expansions were known by Henson and his collaborators, and in many cases, these results went unpublished. The foundational result that makes these expansions amenable to model-theoretic treatment is stated as Theorem \ref{Cor:spec-elem} below, which characterizes elementary equivalence in terms of spectral equivalence; this observation is due to  Eckhardt and Henson based on earlier contributions by Moore. Other fundamental results on expansions of Hilbert spaces by normal operators are scattered throughout the literature, often in the form of special cases, and many times not even appearing in print.  As a result, the authors of this paper decided it was time to write down in a careful and systematic way a fairly comprehensive model-theoretic treatment of these structures.  While several of our results follow the arguments in the literature and others are known to be folklore, the proofs we present here need a small twist, as one has to take into account isolated points of the spectrum and their contribution to the algebraic closure and to forking independence and isolation of types, issues that did not appear when the spectrum is of a special form.  In this way, our paper is partly a survey article and partly a collection of new results.

Let us give a brief historical overview of the subject and the results we present in this paper. Stability of the expansion was known to Henson and Iovino (in the setting of positive bounded formulas) and the failure of $\omega$-stability when the spectrum is rich (one direction of Proposition \ref{prop:omega-stability} below) appears in print in \cite{ArBeCu} for the special case of self-adjoint operators, while the special case when the spectrum is $S^1$ can be found in \cite{BYUsZa}. An argument characterizing $\omega$-stability in terms of the spectrum can be found in \cite{ArBeCu} when the underlying operator is self-adjoint. 
Characterizing forking in these expansions is a small modification of the classical arguments by Chatzidakis-Pillay that appear in \cite{ChPi} and some versions for specific expansions of Hilbert spaces with self-adjoint operators can be found in \cite{ArBe, BYUsZa, BeBu}. 
A proof of superstability of expansions by unitary operators with full spectrum appears in the unpublished note of Ben Yaacov, Usvyatsov, and Zadka \cite{BYUsZa} and for more general generic group actions with amenable discrete groups in \cite{Be}.
The study of the model-theoretic perturbations of the operator, for the special case of unitary operators with full spectrum, was carried out in \cite{BYBe}. Understanding types in terms of the measures they define (a model theoretic interpretation of the spectral decomposition theorem, see Fact \ref{fact: measures} below) first appears in \cite{BYUsZa} for expansions by a unitary operator with full spectrum.

In this paper we will prove many of the known results described above, and, as stated, many of our proofs will be inspired by some of the arguments that have either appeared in the literature or in unpublished notes.  That being said, we also present several new contributions. One such new result worth pointing out is the definability of the adjoint $T^*$ in the structure $(H,T)$ after adding to the theory axioms stating that $T$ is a bounded operator. Most of the results from Section \ref{sec:normal operators} are also new, in particular, Theorem \ref{thm:limit theories} characterizing limits of theories of expansions with normal operators in terms of Gromov-Hausdorff limits of their spectra
and Lemma \ref{lem:Application-definability} establishing the definability of certain spectral projections. Other results from Section \ref{sec:normal operators}, like the theory of the expansion $(H,T)$ being pseudocompact, are most likely folklore.

As usual, we expect the reader to be familiar with continuous logic as presented in \cite{BBHU}. In particular, we assume that the reader is familiar with how Hilbert spaces are treated as structures in continuous logic (see \cite[Section 15]{BBHU}) although we briefly recall the set-up in the next section.  To follow the material in Section \ref{sec:stability}, we also expect the reader to be familiar with basic results from stability theory, such as the stability spectrum, characterizations of forking in terms of the properties of its associated independence relation, and notions such as non-multidimensionality and orthogonality.  A good reference for basic material is \cite{BYUs}, a discussion of independence relations can be found in \cite[Theorems 14.12 and 14.14]{BBHU} (which follows the classical arguments from \cite{Wa}), and non-multidimensionality and orthogonality are defined as in the discrete case \cite{Pi}.
Finally, the last results of Section \ref{sec:stability} require some knowledge of perturbations in the continuous setting; see, for example, \cite{BY-pert, HaIb,HiHy} (our definition of $d_{\pert}$ comes from \cite{HiHy}).

This paper is organized as follows. We introduce the basic results that we will need from functional analysis in Section \ref{sec:prelim}. While our main reference for most of the results is \cite{conway}, we also recommend the reader consult \cite{GePa} for a very readable and elementary proof of the version of the Weyl-von Neumann-Berg theorem that we present as Fact \ref{fact:WNB} below. In Section \ref{sec:normal operators}, we present the basic model theoretic results of the paper: an axiomatization of expansions by bounded normal operators, definability of the adjoint $T^*$ in the expansion $(H,T)$, a characterization of the complete theories of these expansions according to spectral data of the underlying operator, and a characterization of convergence of these complete theories  in terms of the behavior of the spectrum. In Section \ref{sec:QE}, we prove quantifier elimination for expansions of the form $(H,T,T^*)$ when the operator $T$ is normal. In Section \ref{types-categ}, we study types in terms of the measure they define and we use this fact to characterize the logic topology as the weak$^*$-topology of the associated measures and the distance topology (for $1$-types) via a formula reminiscent of the formula defining the total variation metric. In Section \ref{sec:stability}, we show all expansion by normal operators are superstable. Furthermore,  we show the expansion $(H,T)$ is $\omega$-stable if and only if the spectrum of $T$ is at most countable. We also characterize forking independence in the expansion, which allows us to show that all expansions have weak elimination of imaginaries. Finally, we prove all such expansions are $\omega$-stable up to perturbation and admit prime models over arbitrary parametersets, again up to perturbation.  We conclude the paper with a list of some open questions.

We end this introduction with some conventions we use throughout the paper:
\begin{itemize}
    \item   We let $\mathbb S$ and $\mathbb D$ denote the unit circle and unit disk in the complex plane, respectively.  
    \item For a compact subset $K$ of $\mathbb C$, $M(K)$ denotes the space of regular Borel complex measures on $K$. We also let $M(K)_+$ denote the collection of the positive measures in $M(K)$.
    \item For a complex measure $\mu$, we let $\|\mu\|$ denote its total variation.  
    \item When $V$ is a closed subspace of a Hilbert space $H$, we sometimes use the notation $V^\perp$ for the orthogonal complement of $V$ in $H$ and other times we use the more informative notation $H\ominus V$, which makes clear the ambient space in which the orthogonal complement is being computed.
\end{itemize}

\section{Preliminaries on normal operators}\label{sec:prelim}

We begin by recalling some well-known facts about normal operators.  For details and proofs, we suggest \cite{conway}.  Throughout this paper, $H$ denotes a complex Hilbert space and $T:H\to H$ a bounded operator.  We usually assume that the operator norm $\|T\|$ of $T$ is bounded by $1$, the general case being handled simply by rescaling.  The \textbf{adjoint of $T$} is the unique bounded operator $S:H\to H$ satisfying $\langle Tx,y\rangle=\langle x,Sy\rangle$ for all $x,y\in H$; the adjoint of $T$ is denoted by $T^*$.

A bounded operator $T:H\to H$ is called \textbf{normal} if $T$ commutes with its adjoint, that is, if $TT^*=T^*T$; equivalently, $T$ is normal if $\|Tx\|=\|T^*x\|$ for all $x\in H$ (see \cite[Proposition II.2.16]{conway}).  The class of normal operators includes the case of self-adjoint operators (where $T=T^*$) and the class of unitary operators (where $TT^*=T^*T=\operatorname{Id}_H$, the identity operator on $H$).

For a general bounded operator $T:H\to H$, a closed subspace $V$ of $T$ is called \textbf{invariant for $T$} if $T(V)\subseteq V$ and \textbf{reducing for $T$} if it is invariant for both $T$ and $T^*$; in this case, the orthogonal complement $V^\perp$ of $V$ is also invariant under both $T$ and $T^*$ and $(T\! \upharpoonright_V)^*=T^*\! \upharpoonright_V$ as operators on $V$.  An important example of a subspace of $H$ that is reducing for $T$ is any eigenspace $H_\lambda:=\{x\in H \ : \ Tx=\lambda x\}$\footnote{This is really only an eigenspace if $H_\lambda\not=\{0\}$.}; indeed, if $x\in H_\lambda$, then $T^*x=\bar \lambda x\in H_\lambda$.  

In general, a subspace invariant for $T$ need not be reducing for $T$; for example, if $T$ is the unilateral shift on $\ell^2(\mathbb Z)$, then $\ell^2(\mathbb N)$ is invariant for $T$, but not reducing for $T$.  However, in this example note that $T\! \upharpoonright \ell^2(\mathbb N)$ is the usual shift operator on $\ell^2(\mathbb N)$, which is not normal.  Contrast this with the following fact (see, for example, \cite[Exercise IX.9.5]{conway}:

\begin{fact}\label{normalreducing}
Let $T$ be a normal operator on $H$ and let $V$ be a $T$-invariant subspace of $H$.  Then $T\! \upharpoonright V$ is normal if and only if $V$ is reducing for $T$.
\end{fact}

For an arbitrary bounded operator $T:H\to H$, the \textbf{spectrum of $T$} is the set $\sigma(T):=\{\lambda \in \mathbb{C} \ : \ T-\lambda \text{ is not invertible}\}$; it is a nonempty, compact subset of the disk in $\mathbb{C}$ centered at the origin of radius $\|T\|$.  The spectrum $\sigma(T)$ of $T$ always contains the \textbf{approximate eigenvalues} of $T$, which are those $\lambda \in \mathbb{C}$ for which there exist unit vectors $x_n\in H$ with $\|Tx_n-\lambda x_n\|\to 0$ as $n\to \infty$; this of course includes the set of actual eigenvalues of $T$.  A general bounded operator need not have any eigenvalues, even when $N$ is normal (consider, for example, the shift operator on $\ell_2(\mathbb{Z}))$, although it always has approximate eigenvalues.  In fact, every element of the boundary of $\sigma(T)$ is an approximate eigenvalue of $T$; see \cite[Proposition VII.6.7]{conway}.)  For normal operators, the set of approximate eigenvalues of $T$ is precisely the spectrum $\sigma(T)$ of $T$.  Moreover, in this case, $\sigma(T)$ decomposes into two pieces:
\begin{itemize}
    \item The isolated points $\operatorname{Isol}(\sigma(T))$ of $\sigma(T)$, which are then necessarily eigenvalues for $T$ whose corresponding eigenspaces may be finite- or infinite-dimensional.
    \item The accumulation points $\sigma(T)'$ of $\sigma(T)$, which may or may not be eigenvalues of $T$; when they are eigenvalues, the corresponding eigenspaces are infinite-dimensional.
\end{itemize}

Another consequence of the fact that the spectrum of a normal operator coincides precisely with its approximate eigenvalues is that whenever $T:H\to H$ is normal and $V\subseteq H$ is reducing for $T$ (so $T\! \upharpoonright V$ is also normal), then $\sigma(T\! \upharpoonright 
 V)\subseteq \sigma(T)$; in general, the spectrum of the restriction of an operator to a closed invariant subspace need not be contained in the spectrum of the original operator.\footnote{This is probably well-known, but we could not find a precise reference.  For example, one can let $H=\ell^2(\mathbb Z)$, $V:=\{a\in H \ : \ a_i=0 \text{ for all }i<0\}$, and $T$ be the right-shift operator.  Then $\sigma(T)=\mathbb{S}^1$ while $\sigma(T\upharpoonright V)=\mathbb{D}$.}



The following lemma is well-known, but since we were unable to find a reference for it, we include a proof for the sake of completeness.

\begin{lem}\label{lem:countable}
    If $T$ is normal and $\sigma(T)$ is countable, then $H$ is generated by the eigenvectors of $T$.
\end{lem}
\begin{proof}
Let $E$ be the linear closed subspace
spanned by all eigenvectors of $T$ and let $E^\bot$ be its orthogonal complement, so $H=E\oplus E^\bot$. We wish to show that $H=E$. Since $\sigma(T)$ is a non-empty, countable closed subset of $\mathbb{C}$, it is Cantor-Bendixson analyzable and has an isolated point, whence $\dim(E)>0$. Note that $E$ is reducing for $T$. Assume, towards a contradiction, that $E^\perp \neq \{0\}$. Since $\sigma(T\! \upharpoonright {E^\perp}) \subseteq \sigma(T)$, it is at most countable. As above, since $\sigma(T\! \upharpoonright{E^\perp})$ is a non-empty countable closed subset of $\mathbb{C}$, it has an isolated point and thus $(E^\perp,T\! \upharpoonright{E^\perp})$ has an eigenvector, a contradiction to the definition of $E$.
\end{proof}

Once again, the following lemma is probably well-known, but we were unable to find a reference in the literature.

\begin{lem}\label{continuousfunctionalcalculus}
    If $T$ is a normal operator on $H$ and $\lambda\in \mathbb{C}\setminus \sigma(T)$ is such that $\|Tx-\lambda x\|<\epsilon$ for some unit vector $x\in H$, then $d(\lambda,\sigma(T))<\epsilon$.
\end{lem}  

\begin{proof}
    The assumption implies that $\|(T-\lambda I)^{-1}\|>\epsilon$.  On the other hand, by continuous functional calculus (see \cite[Theorem VIII.2.6]{conway}), setting $f(z):=1/(z-\lambda)$ for $z\in \sigma(T)$, one has $\|(T-\lambda I)^{-1}\|=\|f\|_\infty=1/d(\lambda,\sigma(T))$.
    \end{proof}

We will also make heavy use of the Spectral theorem for normal operators, which we now describe.  First, by a \textbf{projection valued measure (PVM)} we mean a function $E$ defined on the measurable sets of some measurable space $(X,\mathcal{B})$ into the set of projections of some Hilbert space $H$ such that:
\begin{enumerate}
    \item $E(\emptyset)=0$ and $E(X)=\operatorname{Id}_H$;
    \item $E(A\cap B)=E(A)E(B)$ for $A,B\in \mathcal{B}$; and
    \item $E(\bigcup_{n=1}^\infty A_n)=\sum_{n=1}^\infty E(A_n)$ whenever $(A_n)_{n\in \mathbb{N}}$ is a family of disjoint elements from $\mathcal{B}$.
\end{enumerate}

For a proof of the next fact, see \cite[Chapter IX, Section 1]{conway}.

\begin{fact}\label{fact: measures}
 Suppose that $E$ is a PVM defined on $(X,\mathcal{B})$ with values in the Hilbert space $H$.  Then:
 \begin{enumerate}
     \item For all $v,w\in H$, the function $\mu_{v,w}$ on $\mathcal{B}$ given by $\mu_{v,w}(A):=\langle E(A)v,w\rangle$ is a countably additive measure on $(X,\mathcal{B})$ and $\|\mu_{v,w}\|\leq \|v\|\|w\|$.  Setting $\mu_v:=\mu_{v,v}$, we have that $\mu_{v}$ is a positive measure with $\mu_{v}(X)=\|v\|^2$.
     \item If $\phi:X\to \mathbb{C}$ is a bounded $\mathcal{B}$-measurable function, then there is a unique bounded operator on $H$, denoted $\int_X \phi dE$, satisfying the following property:  whenever $\{A_1,\ldots,A_n\}$ is a finite partition of $X$ into elements of $\mathcal{B}$ for which $\sup\{|\phi(x)-\phi(x')| \ : \ x,x'\in A_i\}<\epsilon$ for all $i=1,\ldots,n$, then for any choice of $a_i\in A_i$, $i=1,\ldots,n$, we have $\|\int_X \phi dE-\sum_{i=1}^n \phi(a_i)E(A_i)\|<\epsilon$.
     \item For $\mu_{v,w}$ as in (1) and $\int_X \phi dE$ as in (2), we have that $$\left\langle \left(\int_X \phi dE\right) v,w\right\rangle=\int_X \phi d\mu_{v,w}.$$
 \end{enumerate}
\end{fact}

The following is one of the main facts about  normal operators; see, for example \cite[Chapter IX, Section 2]{conway}.

\begin{fact}[Spectral theorem for normal operators]
If $T$ is a normal operator on $H$, then there is a PVM $E$ defined on the Borel subsets of $\sigma(T)$ with values in $H$ such $\int_{\sigma(T)}zdE=T$
\end{fact}
The PVM in the previous theorem is unique modulo some mild requirements; see \cite[Theorem IX.2.2]{conway}.  We refer to this unique PVM as the \textbf{spectral measure associated to $T$} and sometimes refer to $E(A)$ as a \textbf{spectral projection associated to $T$}.  A fact about the spectral measure $E$ associated to $T$ is that $E(\mathcal{O})\not=0$ if $\mathcal{O}$ is a nonempty open subset of $\sigma(T)$.  It will be convenient to view $E$ as being defined on all Borel subsets of $\mathbb{C}$ by defining $E(\mathcal{O})=E(\mathcal{O}\cap \sigma(T))$; as a result, $E(\mathcal{O})=0$ if $\mathcal{O}\cap \sigma(T)=\emptyset$.

The following lemma will be used a couple of times throughout this paper; the model theorist will recognize the connection with definability and, in fact, this lemma will imply that eigenspaces corresponding to isolated points of the spectrum are definable sets (see Lemma \ref{isolateddefinable} below).

\begin{lem}\label{lem:almostnear}
For all $\epsilon,\eta>0$, there is $\delta>0$ such that, whenever $T:H\to H$ is normal, $\lambda\in \sigma(T)$, $U=B(\lambda;\eta)$, and $x\in H$ is a unit vector such that $\|Tx-\lambda x\|<\delta$, then there is $y\in E(U)(H)$ such that $\|x-y\|<\epsilon$.  
\end{lem}

\begin{proof}
We claim that $\delta:=\epsilon\cdot (\eta/2)$ satisfies the conclusion of the lemma.  Write $x=y+z$, with $y\in E(U)(H)$ and $z\in E(U^c)(H)$.  Set $S:=T\upharpoonright E(U)(H)^\perp$, so $\sigma(S)\subseteq \{\lambda'\in K \ : \ |\lambda-\lambda'|\geq \eta\}$.  Take a partition $I_1,\ldots,I_t$ of $\sigma(S)$ such that each $I_i$ has diameter at most $\eta/2$ and choose $\lambda_i\in I_i$ for all $i=1,\ldots,t$.  Then $\|Tz-\sum_i \lambda_i E(I_i)(z)\|<(\eta/2)\|z\|$.  On the other hand, $$\|\lambda z-\sum_i \lambda_i E(I_i)(z)\|^2=\|\sum_i \lambda E(I_i)(z)-\sum_i \lambda_i E(I_i)(z)\|^2\geq \eta^2\|z\|^2.$$  Finally, note that, since $E(U)(H)$ and $E(U^c)(H)$ are $T$-invariant, we have $$\|Tx-\lambda x\|^2=\|Ty-\lambda y\|^2+\|Tz-\lambda z\|^2,$$ and thus $\|Tz-\lambda z\|^2\leq \|Tx-\lambda x\|^2<\delta^2$.  Putting this all together, we arrive at
$$\eta \|z\|\leq \|\lambda z-\sum_i \lambda_iE(I_i)(z)\|\leq \|\lambda z-Tz\|+\|Tz-\sum_i \lambda_i E(I_i)(z)\|<\delta+(\eta/2)\|z\|.$$  It follows that $\|z\|<\delta/(\eta/2)=\epsilon$.
\end{proof}

Crucial to the classification of normal operators up to elementary equivalence is the following notion:

\begin{defn}
Suppose that $T_i:H_i\to H_i$ are bounded operators for $i=1,2$.  We say  $T_1$ and $T_2$ are \textbf{spectrally equivalent} if:
\begin{enumerate}
    \item $\sigma(T_1)=\sigma(T_2)$, and
    \item for each isolated point $\lambda\in \operatorname{Isol}(\sigma(T_1))$, the dimensions of the eigenspaces $H_{1,\lambda}:=\ker(T_1-\lambda)$ and $H_{2,\lambda}:=\ker(T_2-\lambda)$ are either the same finite number or are both of infinite dimension.
\end{enumerate}
\end{defn}

The following fact is a combination of results that can be found in \cite{GePa}, \cite{H1}, and \cite{H3}.

\begin{fact}\label{fact:WNB}
Suppose that $T_i:H_i\to H_i$ is a normal operator for $i=1,2$ with associated spectral measures $E_1$ and $E_2$, viewed as defined on all Borel subsets of $\mathbb{C}$.  Then the following are equivalent:
\begin{enumerate}
    \item $T_1$ and $T_2$ are \textbf{approximately unitarily equivalent}, meaning that there is a sequence of unitary operators $U_n:H_1\to H_2$ such that $$\lim_{n\to \infty}\|U_nT_1U_n^*-T_2\|=0.$$
    \item For each open set $\mathcal O\subseteq \mathbb C$, the spectral projections $E_1(\mathcal O)$ and $E_2(\mathcal O)$ have the same rank.
\end{enumerate}
If, in addition, the spaces $H_1$ and $H_2$ are separable, (1) and (2) are equivalent to:
\begin{enumerate}
    \item[(3)] $T_1$ and $T_2$ are spectrally equivalent.
\end{enumerate}
\end{fact}

Some authors derive (the nontrival direction of) Fact \ref{fact:WNB} as a consequence of the Weyl-von Neumann-Berg theorem; see, for example \cite[Section II.4]{Da}.  For that reason, the previous fact is sometimes often referred to as the Weyl-von Neumann-Berg theorem.


\section{Theories of normal operators}\label{sec:normal operators}

Throughout, we let $L:=\{0,-,\dot2, \frac{x+y}{2},i,\| \ \ \|,T\}$ be the language of unit balls of Hilbert spaces over the complex numbers expanded by a unary function symbol $T$ with modulus of uniform continuity the identity function. We write $\Sigma_0$ for the $L$-theory whose models are of the form $(H,T)$, where $H$ is (the unit ball of a) Hilbert space and $T$ is a bounded operator on $H$ with $\|T\|\leq 1$. Note that $T$ sends the unit ball to the unit ball and the action of $T$ is determined on the whole Hilbert space by the restriction of its action to the unit ball. Note also that $\Sigma_0$ is a universal theory. According to this language, all quantifiers range over the unit ball. That being said, we will be very liberal when dealing with actual formulas and use expressions like $\langle v,w \rangle$ even though they have values in $[-1,1]$
and sometimes, for clarity, we will write $d(v,w)=\sqrt{2}$ for orthonormal vectors $v,w$ even though it formally has the value $d(v,w)=\|\frac{v-w}{2}\|=\sqrt{2}/2$. When considering more general bounded operators $T$, one can 
analyze them with this approach by working with $T'=T/\|T\|$ instead of $T$.

To axiomatize infinite-dimensional Hilbert spaces, we need to add the scheme that says that there are infinitely many orthonormal vectors:
$$\inf_{x_1}\cdots\inf_{x_n} \max_{1\leq i\leq j\leq n}|\langle x_i,x_j\rangle-\delta_{ij}|=0.$$
Note that the scheme consists of existential sentences. Several subspaces of the expansions under consideration will be finite-dimensional, but \emph{the structure itself will always be infinite-dimensional unless we explicitly state otherwise}.

\begin{prop}\label{normalaxioms}
There is an $\forall\exists$-axiomatizable $L$-theory $\Sigma_{\operatorname{normal}}$ whose models are precisely the models $(H,T)$ of $\Sigma_0$ with $T$ a normal operator.
\end{prop}

\begin{proof}
For each $n\geq 1$, consider the $L$-sentence $\sigma_n$ given by
$$\sup_x\sup_y\min(1/n\dminus \sup_z\left|\langle Tz,x\rangle -\langle z,y\rangle\right|,|\|Tx\|-\|y\||\dminus 1/n).$$  We claim that $(H,T)\models \sigma_n$ for all $n\geq 1$ if and only if $T$ is normal.  Indeed, first suppose that $(H,T)\models \sigma_n$ for all $n\geq 1$ and let $x\in H$ be such that $\|x\|\leq 1$.  Since $\sup_z|\langle Tz,x\rangle-\langle z,T^*x\rangle|=0$, letting $y=T^*x$ in the above axiom, we have that $|\|Tx\|-\|T^*x\||\leq 1/n$ for all $n\geq 1$, whence $\|Tx\|=\|T^*x\|$; since $x$ was arbitrary, we see that $T$ is normal.  Conversely, suppose that $T$ is normal and let $x,y\in H$ be such that $\|x\|,\|y\|\leq 1$ and $\sup_z |\langle Tz,x\rangle-\langle z,y\rangle|<1/n$.  It follows that $\sup_z|\langle z,T^*x-y\rangle|<1/n$, whence $\|T^*x-y\|<1/n$ and thus $|\|T^*x\|-\|y\||<1/n$.  Since $T$ is normal, $\|T^*x\|=\|Tx\|$, whence $|\|Tx\|-\|y\||<1/n$, as desired.
\end{proof}

\begin{rem}
By Fact \ref{normalreducing}, one cannot expect to axiomatize the class of Hilbert spaces expanded by normal operators by a universal set of axioms. It is also easy to see that one can extend a Hilbert space expanded by a normal operator to a larger Hilbert space expanded with an operator that is not normal, for example by taking a direct sum with a Hilbert space expanded by an operator which is not normal; it follows that the class cannot be axiomatized by existential axioms either. On the other hand, the class is clearly closed under unions of chains, which implies that an axiomatization, should it exist, can be assumed to consist of $\forall\exists$-sentences.
\end{rem}

The idea behind the proof of Proposition \ref{normalaxioms} can be used to show the following.  We use the terminology around definability established in \cite{goldbring}.

\begin{prop}\label{graph}
The $\Sigma_0$-functor which assigns to a model $(H,T)$ of $\Sigma_0$ the graph $\Gamma(T^*)$ of $T^*$ is a $\Sigma_0$-definable set.
\end{prop}

\begin{proof}
Let $\varphi(x,y)$ be the formula $\sup_z |\langle Tz,x\rangle-\langle z,y\rangle|$.  The zeroset of $\varphi$ is the graph of $T^*$.  We must show that $\varphi$ is an almost-near formula.  Towards that end, suppose that $\sup_z |\langle Tz,x\rangle-\langle z,y\rangle|<\epsilon$.  Then, as in the proof of the previous proposition, we have $\|T^*x-y\|<\epsilon$ and thus $d((x,y),(x,T^*x))<\epsilon$.
\end{proof}

\begin{rem}

\

\begin{enumerate}
    \item Proposition \ref{graph} allows for an alternative proof of Proposition \ref{normalaxioms}.  Indeed, one can simply add to $\Sigma_0$ the single axiom $\sup_{(x,y)\in \Gamma(T^*)} |\|Tx\|-\|y\||=0$.  Note that the sentence appearing in this axiom is a $\Sigma_0$-formula.
    \item The definition of $\Gamma(T^*)$ appearing in the proof of Proposition \ref{graph} is universal.  In general, the graph of $T^*$ cannot be defined by a quantifier-free $\Sigma_0$-formula.  Indeed, if it were, then any inclusion $(H_1,T_1)\subseteq (H_2,T_2)$ between models of $\Sigma_0$ would necessarily satisfy that $T_1^*=T_2^*\upharpoonright \! H_1$ and thus $H_1$ would be $T_2^*$-invariant.  However, as pointed out in the previous section, this is not true in general.  That being said, by Fact \ref{normalreducing}, if $(H_1,T_1),(H_2,T_2)\models \Sigma_{\operatorname{normal}}$ are such that $(H_1,T_1)\subseteq (H_2,T_2)$, we do have that $T_1^*=T_2^*\! \upharpoonright H_1$.
\end{enumerate}
\end{rem}

Lemma \ref{lem:almostnear} above immediately implies the following:

\begin{prop}\label{isolateddefinable}
Take $(H,T)\models \Sigma_{\operatorname{normal}}$ and $\lambda\in \operatorname{Isol}(\sigma(T))$ with associated eigenspace $H_\lambda$. Then $H_\lambda$ is a quantifier-free $0$-definable subset of $(H,T)$.  Moreover, the definition of $H_\lambda$ depends only on a lower bound on $\operatorname{dist}(\lambda,\sigma(T)\setminus \{\lambda\})$.
\end{prop}

\begin{lem}
Assume that $\lambda\in \operatorname{Isol}(\sigma(T))$ is such that $\dim(H_\lambda)=n\in \mathbb N$. Then for any $x\in H_\lambda$, we have $x\in \acl(\emptyset)$.  More generally, if $\{\lambda_i:i\in I\}$ are distinct elements of $\operatorname{Isol}(\sigma(T))$ with $\dim(H_{\lambda_i})\in \mathbb N$ for all $i\in I$, and $x\in H$ satisfies $x=\sum_{i\in I}P_{H_{\lambda_i}}(x)$, then $x\in \acl(\emptyset)$.
\end{lem}

\begin{proof}
    
Let $(H',T')\succeq (H,T)$ 
be sufficiently saturated. 
By Proposition \ref{isolateddefinable}, we have $\dim(H'_\lambda)=n$.
Since the unit ball in a finite dimensional Hilbert space is compact, the orbit of any $x\in H_\lambda$ is a compact subset in $(H',T')$, and thus $x\in \acl(\emptyset)$.  For the second statement, if $\{\lambda_i:i\in I\}$ are different isolated points in $\sigma(T)$ with $\dim(H_{\lambda_i})\in \mathbb N$ for each $i\in I$, then $P_{H_{\lambda_i}}(x)\in \acl(\emptyset)$ for all $x\in H$ and $i\in I$. If $x=\sum_{i\in I}P_{H_{\lambda_i}}(x)$, then $x\in \dcl(P_{H_{\lambda_i}}(x) :i\in I)\subseteq \acl(\emptyset)$.
\end{proof}

\begin{defn}
Given $(H,T)\models \Sigma_{\operatorname{normal}}$, we write $H_{\fin,T}$ for the closed subspace of $H$ generated by the spaces $H_\lambda$ with 
$\lambda\in \operatorname{Isol}(\sigma(T))$ satisfying $\dim(H_\lambda)\in \mathbb N$.  In other words, if $I=\{\lambda\in \operatorname{Isol}(\sigma(T))\ : \dim(H_\lambda)\in \mathbb N\}$, then $H_{\fin,T}=\bigoplus_{\lambda \in I} H_{\lambda}$. \end{defn}

  We note the following:

\begin{cor}\label{cor:Hfin-invariant}
The space $H_{\fin,T}$ is a closed subspace which is reducing for $T$.    
\end{cor}

We write $T_{\fin}$ for $T\upharpoonright \! H_{\fin,T}$ and note that, since $H_{\fin,T}\subseteq \acl(\emptyset)$, we have that $(H_{\fin},T_{\fin})$ is a substructure of any model of $\operatorname{Th}(H,T)$.

We now work towards characterizing the completions of $\Sigma_{\operatorname{normal}}$.

\begin{lem}\label{lem:spectequivalent}
Suppose that $(H_1,T_2)$ and $(H_2,T_2)$ are models of $\Sigma_{\operatorname{normal}}$ with the same existential theory.  Then $T_1$ and $T_2$ are spectrally equivalent.
\end{lem}

\begin{proof}
Recall that every element of the spectrum of a normal operator is an approximate eigenvalue.  Given $\lambda\in \mathbb C$, to say that $\lambda$ is an approximate eigenvalue of the normal operator $T$ can be captured by an existential condition, namely by $\inf_x \max(1\dminus \|x\|,\|Tx-\lambda x\|)=0$.  It follows that $\sigma(T_1)=\sigma(T_2)$.

Now suppose that $\lambda\in \operatorname{Isol}(\sigma(T_1))$.  Suppose also that $\ker(T_1-\lambda)$ has dimension at least $n$.  Then
$$(H_1,T_1)\models \inf_{x_1}\cdots\inf_{x_n} \max\left(\max_{1\leq i\leq n}\|Tx_i-\lambda x_i\|,\max_{1\leq i\leq j\leq n}|\langle x_i,x_j\rangle-\delta_{ij}|\right)=0.$$  By assumption, the same condition is satisfied in $(H_2,T_2)$, meaning that there are $n$ vectors in $H_2$ which are approximate eigenvectors corresponding to $\lambda$ that are almost orthonormal.  By Proposition \ref{isolateddefinable}, there are $n$ eigenvectors for $T_2$ with eigenvalue $\lambda$ that are almost orthonormal; this is enough to conclude that the dimension of $\ker(T_2-\lambda)$ is at least $n$.  Switching the roles of $T_1$ and $T_2$ yields the desired equality.
\end{proof}

The previous lemma and Fact \ref{fact:WNB} allow us to describe the completions of $\Sigma_{\operatorname{normal}}$, a result first observed by Henson.

\begin{thm}\label{Cor:spec-elem}
If $T_1$ and $T_2$ are normal operators on Hilbert spaces $H_1$ and $H_2$, the following are equivalent:
\begin{enumerate}
    \item $(H_1,T_1)$ and $(H_2,T_2)$ are elementarily equivalent.
    \item $(H_1,T_1)$ and $(H_2,T_2)$ have the same existential theory.
    \item $(H_1,T_1)$ and $(H_2,T_2)$ are spectrally equivalent.
\end{enumerate}
If, in addition, $H_1$ and $H_2$ are separable, these conditions are equivalent to:
\begin{enumerate}
   \item[(4)] $(H_1,T_1)$ and $(H_2,T_2)$ are approximately unitarily equivalent.
\end{enumerate}
\end{thm}

\begin{proof}
    The only direction needing some justification is the implication (3) implies (1); but this follows from the previous fact by passing to separable elementary substructures and noting that approximately unitarily equivalent operators have isomorphic ultrapowers, and are thus elementarily equivalent. 
\end{proof}

It follows that the completions of $\Sigma_{\operatorname{normal}}$ are determined by two pieces of data, namely the spectrum of the operator and the dimensions of the eigenspaces corresponding to isolated points of the spectrum.

\begin{prop}
Fix a compact subset $K\subseteq \mathbb D$ and a function $m:\isol(K)\to \mathbb N\cup\{\infty\}$.
\begin{enumerate}
    \item There is an $\forall\exists$-theory $\Sigma_{K}\supseteq \Sigma_{\operatorname{normal}}$ whose models are those models $(H,T)$ of $\Sigma_{\operatorname{normal}}$ with $\sigma(T)=K$.
    \item There is an $\forall\exists$-theory $\Sigma_{K,m}\supseteq \Sigma_K$ whose models are those models $(H,T)$ of $\Sigma_K$ for which the multiplicity of the eigenvalues of isolated points is given by $m$.
\end{enumerate}
\end{prop}

\begin{proof}
For (1), we add existential axioms stating that each $\lambda\in K$ has approximate eigenvectors, which is enough to conclude that $K$ is contained in the spectrum of all models.  To ensure that all models have spectrum exactly $K$, fix $\lambda\notin K$ and set $\delta:=\operatorname{dist}(\lambda,K)$.  We then add the universal axioms $\sup_x (\delta \dminus \|Tx-\lambda x\|)=0$; by Lemma \ref{continuousfunctionalcalculus}, these axioms do indeed ensure that $\sigma(T)=K$.

For (2), first suppose that $\lambda\in \operatorname{Isol}(K))$ and $m(\lambda)=n\in \mathbb N$. Then by Proposition \ref{isolateddefinable} the eigenspace $H_{\lambda}$ is definable, and can thus be quantified over. We then add the following axioms to ensure that all $m(\lambda)=n$ in all models: 

$$\inf_{x_1\in H_{\lambda}}\cdots \inf_{x_n\in H_\lambda }\max\left(\max_{1\leq i \leq j\leq n}|\langle x_i,x_j\rangle-\delta_{ij}|,\sup_{y\in H_{\lambda}} \|y-\sum_{k=1}^n \langle y,x_k\rangle x_k\|\right)=0.$$ 

If $\lambda\in \isol(K)$ and $m(\lambda)=\infty$, then add the following scheme of axioms indexed by $n\geq 1$ to ensure that $m(\lambda)=\infty$ in all models:
$$\inf_{x_1\in H_{\lambda}}\cdots \inf_{x_n\in H_\lambda }\left(\max_{1\leq i \leq j\leq n}|\langle x_i,x_j\rangle-\delta_{ij}|\right)=0.\qedhere$$ 

Perhaps it is not entirely obvious that these axioms are $\forall\exists$-axioms, but it is clear that a union of a chain of models of $\Sigma_{K,m}$ is once again a model of $\Sigma_{K,m}$, whence the theory is indeed $\forall\exists$-axiomatizable.

\end{proof}

\begin{prop}
For each compact $K\subseteq \mathbb D$ and each function $m:\isol(K)\to \mathbb N\cup\{\infty\}$, the theory $\Sigma_{K,m}$ is satisfiable.
\end{prop}

\begin{proof}
Extend $m$ to a function defined on all of $K$ by setting $m(\lambda)=\infty$ for all $\lambda\in K\setminus \isol(K)$.  For each $\lambda \in K$, let $H_\lambda$ be a Hilbert space of dimension $m(\lambda)$ and define $T_\lambda(x)=\lambda x$ for all $x\in H_\lambda$.  Then $(H,T):=\bigoplus (H_\lambda,T_\lambda)$ is a model of $\Sigma_{K,m}$. 
\end{proof}

It follows that the completions of $\Sigma_{\operatorname{normal}}$ are precisely those theories of the form $\Sigma_{K,m}$, or simply $\Sigma_K$ if $K$ is perfect.

Let $\mathcal X$ denote the set of completions of $\Sigma_{\operatorname{normal}}$.  We equip $\mathcal X$ with the induced topology it receives from being a subspace of the space of complete $L$-theories equipped with the logic topology, making it a compact Hausdorff space.  A description of the topology of $\mathcal X$ can be given in terms of characterizations of ultralimits: if $((H_i,T_i) \ : i\in I)$ is a family of models of $\Sigma_{\operatorname{normal}}$ and $\mathcal U$ is an ultrafilter on $I$, then $\lim_\mathcal U \operatorname{Th}(H_i,T_i)=\operatorname{Th}(\prod_\mathcal U (H_i,T_i))$.  

We now aim to give an alternate description of the topology on $\mathcal X$ in terms of spectral data.  Fix a family $\Sigma_{K_i,m_i}$ of completions of $\Sigma_{\operatorname{normal}}$ and an ultrafilter $\mathcal U$ on $I$.  Note that, for any $\lambda\in \mathbb C$, we have that the sequence $(\lim_{\mathcal{U}}(m_i(B(\lambda;1/k))))_k$ is a non-increasing sequence from $\mathbb N\cup\{\infty\}$.

\begin{thm}\label{thm:limit theories}
For a family $(\Sigma_{K_i,m_i}:i\in I)$ of completions of $\Sigma_{\operatorname{normal}}$ and an ultrafilter $\mathcal U$ on $I$ with $\lim_{\mathcal U}\Sigma_{K_i,m_i}=\Sigma_{K,m}$, we have:
\begin{enumerate}
    \item $K=\lim_{\mathcal U}K_i$, and
    \item For each $\lambda\in \operatorname{Isol}(K)$, we have $m(\lambda)=\inf_k \lim_{\mathcal{U}}(m_i(B(\lambda;1/k)))$.
\end{enumerate}
\end{thm}

\begin{proof}
Without loss of generality, we may assume that $\mathcal U$ is a non-principal ultrafilter.
For each $i\in I$, take models $(H_i,T_i)$ of $\Sigma_{K_i,m_i}$.  It suffices to show that $(H,T):=\prod_\mathcal U (H_i,T_i)$ is a model of $\Sigma_{K,m}$, where $K$ and $m$ are defined by the formulas (1) and (2) above.  We first show that $K=\lim_\mathcal U K_i$.  Fix $\epsilon>0$; we need to show that $d(\sigma(T),K_i)<  \epsilon$ for $\mathcal U$-almost all $i\in I$.  Fix that $\lambda\in \sigma(T)$. Since the structure $(H,T)$ is $\aleph_1$-saturated, there is $v\in H$ with $\|v\|=1$ and $Tv=\lambda v$.  Writing $v=(v_i)_\mathcal U$, we see that $\|T_iv_i-\lambda v_i\|<\epsilon/2$ for $\mathcal U$-almost all $i\in I$. We may assume that $\|v_i\|=1$ for all $i$ and thus for those indices satisfying $\|T_iv_i-\lambda v_i\|<\epsilon/2$, by Lemma \ref{continuousfunctionalcalculus}, it follows that $d(\lambda,K_i)<\epsilon/2$.  Fixing a finite $\epsilon/2$-net $\lambda_1,\ldots,\lambda_t\in \sigma(T)$, we have, for $\mathcal U$-almost all $i\in I$ and all $j=1,\ldots,n$, that $d(\lambda_j,K_i)<\epsilon/2$ and thus $\sup\{d(\lambda,K_i) \ : \ \lambda\in \sigma(T)\}<\epsilon$ for $\mathcal U$-almost all $i\in I$. Now suppose, towards a contradiction, that $A_{\epsilon}:=\{i\in I : d(\lambda_i,\sigma(T))\geq \epsilon \text{ for some }\lambda_i\in K_i\}\in \mathcal{U}$. Fix positive real numbers $(r_i:i\in I)$ such that $\lim_{\mathcal{U}}r_i=0$. For $i\in A_{\epsilon}$, take $v_i\in H_i$ with $\|v_i\|=1$ such that $\|T_iv_i-\lambda_i v_i\|<r_i$; for $i\notin A_\epsilon$ choose any $v_i\in H_i$ with $\|v_i\|=1$. Setting $\lambda:=\lim_\mathcal U \lambda_i$, we have that $Tv=\lambda v$, where $v=(v_i)_\mathcal U$.  Then $\lambda\in \sigma(T)$, contradicting the fact that $d(\lambda,\sigma(T))\geq \epsilon$.  It follows that $d(\sigma(T),K_i)< \epsilon$ for $\mathcal U$-many $i\in I$, as desired.

For (2), take $\lambda\in \operatorname{Isol}(K)$ and $\eta>0$ such that $d(\lambda,K\setminus \{\lambda\})>\eta$.  First assume that $m(\lambda)\geq N$ and take orthonormal eigenvectors $x_1,\ldots,x_N\in H$ with eigenvalue $\lambda$.  Write each $x_j=(x_j^i)_\mathcal U$, where, without loss of generality, each set $\{x^i_1,\ldots,x^i_n\}$ forms an orthonormal set in $H_i$.  Take $k\in \mathbb N$ such that $1/k<\eta/3$. For $i\in I$, consider the open set $$U_i=\{B(\rho;1/k)\ : \ \rho \in K_i\setminus B(\lambda;1/k)\}.$$  Setting $A_k:=\{i\in I: \ B(\lambda;1/k)\cap U_i=\emptyset\}$, we claim that, $A_k\in \mathcal U$.  Indeed, for $\mathcal U$-almost all $i\in I$, we have that $d(K,K_i)<\frac{1}{k}$.  Suppose $B(\lambda;1/k)\cap B(\rho;1/k)\not=\emptyset$, where $d(\rho,\lambda)\geq 1/k$.  Take $\rho'\in K$ such that $d(\rho,\rho')<1/k$; it follows that $\rho'\not=\lambda$ and $d(\lambda,\rho')<3/k<\eta$, contradicting the choice of $\eta$.
For any $i\in A_k$ and any $y^i\in H_i$, we have
$$\|y^i\|^2=\|E(B(\lambda;1/k))(y^i)\|^2+\|E(U_i)(y^i)\|^2.$$

Take $\epsilon>0$ sufficiently small such that whenever $\{x_1,\dots,x_N\}$ and $\{y_1,\dots,y_N\}$ are vectors in a Hilbert space satisfy $\langle x_i,x_j \rangle=\delta_{ij}$ and $\|x_i-y_i\|<\epsilon$, then $\dim(\{y_1,\dots,y_N\})=N$. (Note that $\epsilon$ only depends on $N$.) Fix $j\in \{1,\ldots N\}$; since  $\lim_{\mathcal{U}} T_ix^i_j=\lambda x^i_j$, Lemma \ref{lem:almostnear} implies that  $\lim_{\mathcal{U}}\|E(B(\lambda;1/k))(x^i_j)\|^2=1$. Setting $B_{\epsilon}:=\{i\in I: \|E_i(B(\lambda;1/k)(x^i_j)\|>1-\epsilon \text{ for all }j=1,\ldots,N\}$, we have $B_\epsilon\in \mathcal U$. For $j=1,\ldots,N$, set $y^i_j=E_i(B(\lambda;1/k)(x^i_j)$.
For $i\in B_\epsilon$ and $j=1,\ldots,N$, we have $d(x^i_j,y^i_j)<\epsilon$. By the way we chose $\epsilon$, it follows that $\dim(\{y^i_1,\dots,y^i_N\})=N$. As a result, we have $\lim_{\mathcal{U}}(m_i(B(\lambda;1/k)))\geq N$; since the only condition on $k$ was that $1/k<\eta/3$, we also conclude $\inf_k \lim_{\mathcal{U}}(m_i(B(\lambda;1/k)))\geq N$.

Assume now that $\inf_{k}\lim_{\mathcal{U}}(m_i(B(\lambda;1/k)))\geq N$.  Let $\epsilon$ be as in the previous paragraph.  Let $\delta=\delta(\eta,\epsilon)>0$ be as in Lemma \ref{lem:almostnear}.
 Let $k>0$ be such that $1/k<\delta$ and  $\lim_{\mathcal{U}}(m_i(B(\lambda;1/k)))\geq N$.  For $\mathcal U$-almost all $i\in I$, take orthonormal $x_1^i,\ldots,x_N^i\in E_i(B(\lambda;1/k))(H_i)$. For each of these indices we have $\|T_ix^i_j-\lambda x^i_j\|\leq 1/k<\delta$; for the other indices, choose any $x_1^i,\ldots,x_N^i\in H_i$.  Set $x_j=(x^i_j)_{\mathcal U}$. Then $\{x_1,\ldots,x_N\}$ form an orthonormal set satisfying  $\|Tx_j-\lambda x_j\|\leq 1/k<\delta$ for each $j=1,\ldots,N$. By Lemma \ref{lem:almostnear}, 
 there are $y_1,\dots,y_N\in H_{\lambda}$ such that
 $\|x_j-y_j\|<\epsilon$. By the way we chose $\epsilon$, we have
 $\dim(\{y_1,\dots,y_N\})=N$, and thus 
 $m(\lambda)\geq N$.
\end{proof}

\begin{cor}\label{cor: pseudocompact}
   For each compact $K\subseteq \mathbb D$ and each function $m:\isol(K)\to \mathbb N\cup\{\infty\}$, the theory $\Sigma_{K,m}$ is pseudocompact. 
\end{cor}

\begin{proof}
Let $(\lambda_n)_{n\in \alpha}$ enumerate $\isol(K)$, where $\alpha\in \omega\cup\{\omega\}$, $\alpha$ being chosen depending on whether the family of isolated points is finite or countably infinite.  For each $k\geq 1$, let $(\rho_i^k)_{i=1,\ldots,\ell(k)}$ enumerate a finite $1/k$-net in $K'$.  For each $k\geq 1$ and $n\leq \min(k,\alpha)$, let $H_{n,k}$ have dimension $\min\{k,m(\lambda_n)\}$  and let $T_{n,k}(x)=\lambda_nx$ for all $x\in H_{n,k}$.  For each $i=1,\ldots,\ell(k)$, let $H'_{i,k}$ be an one-dimensional Hilbert space with $T'_{i,k}(x)=\rho_i^kx$ for all $x\in H'_{i,k}$.  Let $H_k=(\bigoplus_{m=1}^{\min k,\alpha} H_{m,k})\oplus (\bigoplus_{i=1}^{\ell(k)}H'_{i,k})$ and $T_k=(\bigoplus_{m=1}^{\min (k,\alpha)} T_{m,k})\oplus (\bigoplus_{i=1}^{m(k)}T'_{i,k})$. Let $\mathcal U$ be a nonprincipal ultrafilter on $\mathbb{N}$.  It follows that $(H,T):=\prod_{\mathcal U}(H_k,T_k)$ is a model of $\Sigma_{K,m}$.  Since each of the spaces $H_{m,k}$ and $H'_{i,k}$ is finite-dimensional, it follows that $\Sigma_{K,m}$ is pseudocompact.
\end{proof}

We now establish a ``continous functional calculus'' for the space $\mathcal X$.  We first need some preliminary lemmas. 

\begin{lem}\label{definablefunctionalcalculus}
Suppose that $f:\mathbb D\to \mathbb D$ is a continuous function.  Let $\Gamma_f$ be the $\Sigma_{\operatorname{normal}}$-functor which maps $(H,T)$ to the graph of $f(T)$.  Then $\Gamma_f$ is a $\Sigma_{\operatorname{normal}}$-definable set. 
\end{lem}

\begin{proof}
It suffices to note that, if $(p_n)$ is a sequence of $*$-polynomials that converge uniformly to $f$, then $d(f(T)(x),y)=\lim_n \|p_n(T,T^*)(x)-y\|$.
\end{proof}

As a nice application of the previous lemma, we have the following:

\begin{lem}\label{lem:Application-definability}
Fix compact $K\subseteq \mathbb D$ and suppose that there are nonempty closed subsets $K_1,K_2\subseteq K$ such that $K=K_1\cup K_2$ and for which there is $\epsilon>0$ such that $d(\lambda_1,\lambda_2)>\epsilon$ for all 
$\lambda_1\in K$ and all $\lambda_2\in K_2$. Then the projections $E(K_i)$ are $\Sigma_K$-definable for $i=1,2$.
\end{lem}

\begin{proof}
Since $K_1$ and $K_2$ are uniformly separated by $\epsilon$, there is a continuous function $f:\mathbb D\to \mathbb D$ satisfying $f(\lambda)=1$ for all $\lambda\in K_1$ and $f(\lambda)=0$ for all $\lambda\in K_2$.  By Lemma \ref{definablefunctionalcalculus}, $f(T)$ is $\Sigma_{\operatorname{normal}}$-definable.  But in models of $\Sigma_K$, $f(T)=E(K_1)$.  An analogous argument shows that $E(K_2)$ is $\Sigma_K$-definable.
\end{proof}

\begin{rem}
When $K_1$ and $K_2$ and $\epsilon>0$ are as above, we call them $\epsilon$-separated.
The previous lemma gives another proof of Proposition \ref{isolateddefinable}: any point $\lambda\in \text{Isol}(K)$ is $\epsilon$-separated from $K\setminus\{\lambda\}$ for some $\epsilon>0$, whence $E(\{\lambda\})=P_{H_\lambda}$ is definable.
The previous lemma also gives a more conceptual proof of part 2 of Theorem \ref{thm:limit theories}. Indeed, if $\lambda\in \text{Isol}(K)$ is $\epsilon$-separated from $K\setminus\{\lambda\}$ and $\lim_{\mathcal U}K_i=K$, we have that $B(\lambda;\epsilon/3)$ is 
$\epsilon/3$-separated from $K_i\setminus B(\lambda;\epsilon/3)$ for $\mathcal{U}$-almost all $i\in I$. Thus the projection $E(B(\lambda;\epsilon/3))$ can approximated by the same sequence of polynomials $(p_n(T,T^*))$ for $\mathcal{U}$-almost all $i\in I$, that is, it is uniformly definable in the family of theories $\Sigma_{K,m}$, $(\Sigma_{K_i,m_i}: i\in I_0)$ for some $I_0\in \mathcal{U}$. If $$(H,T)\models \dim(E(B(\lambda;\epsilon/3)(H))=m(\lambda),$$ then $$(H_i,T_i)\models \dim(E(B(\lambda;\epsilon/3)(H_i))=m(\lambda)$$ for $\mathcal{U}$-almost all $i\in I$.
\end{rem}

We now show that the continuous functional calculus preserves elementary equivalence.

\begin{lem}\label{continuouspreservesee}
Suppose that $f:\mathbb D\to \mathbb D$ is a continuous function and $(H_1,T_1)\equiv (H_2,T_2)$.  Then $(H_1,f(T_1))\equiv (H_2,f(T_2))$.
\end{lem}

\begin{proof}
This follows from immediately Lemma \ref{definablefunctionalcalculus}.  We also provide an alternative proof.  Since $T_1$ and $T_2$ are spectrally equivalent, we may take a sequence $(U_n)$ of unitary maps such that $\|T_2-U_nT_1U_n^*\|\to 0$. Then $$\|f(T_2)-U_nf(T_1)U_n^*\|=\|f(T_2)-f(U_nT_1U_n^*)\|\leq \|f\|_\infty \|T_2-U_nT_1U_n^*\| \to 0,$$ showing that $f(T_1)$ and $f(T_2)$ are approximately unitarily equivalent.
\end{proof}

\begin{rem}
Lemma \ref{continuouspreservesee} cannot be generalized to the case that $f$ is merely a bounded Borel function.  Indeed, if $f=\chi_{\{\lambda\}}$ for some $\lambda\in \sigma(T_1)=\sigma(T_2)$, then $f(T_i)=E_i(\{\lambda\})$ for $i=1,2$; since it can be the case that $E_1=0$ while $E_2\not=0$, we cannot conclude that $(H_1,f(T_1))\equiv (H_2,f(T_2))$.
\end{rem}

\begin{rem}
While we assume that our continuous function $f$ takes values in $\mathbb D$ (so that the operators $f(T)$ also have operator norm at most $1$, and thus follow our current convention), the truth of the previous lemma does not rely on this assumption.  Indeed, suppose that $f:\mathbb D\to \mathbb C$ is a continuous function (not necessarily mapping into $\mathbb D$) and $(H_1,T_1)\equiv (H_2,T_2)$.  Then the conclusion $(H_1,f(T_1))\equiv (H_2,f(T_2))$ still holds.
To see this, note that, since the map $f$ is continuous on a compact set, the image belongs to $B(0;n)\subseteq \mathbb{C}$ for some $n$ and the result follows as in the proof of Lemma \ref{continuouspreservesee}, but now $\|f\|_{\infty}\leq n$
and the resulting normal operator now has norm bounded by $n$.
\end{rem}

By Lemma \ref{continuouspreservesee}, a continuous function $f:\mathbb D\to \mathbb D$ induces a function $\tilde{f}$ on the space $\mathcal X$ of completions of $\Sigma_{\operatorname{normal}}$.

\begin{prop}
The function $\tilde{f}:\mathcal X\to \mathcal X$ is continuous.
\end{prop}

\begin{proof}
Let $((H_i,T_i):i\in I)$ be a collection of models of $\Sigma_{\operatorname{normal}}$, $\mathcal U$ an ultrafilter on $I$, and $(H,T):=\prod_{\mathcal U}(H_i,T_i)$; it suffices to show that $\prod_{\mathcal U}(H_i,f(T_i))=(H,f(T))$.  However, this follows immediately from Lemma \ref{definablefunctionalcalculus}.
\end{proof}

We end this section with one further general observation about the theories under discussion.

\begin{prop}
  The theories $\Sigma_{\operatorname{normal}}$, $\Sigma_K$, and $\Sigma_{K,m}$ have the joint embedding property and the amalgamation property.
\end{prop}

\begin{proof}
   We only treat the case of the theory $\Sigma_{K,m}$, the other cases being simpler.  We first establish the joint embedding property. Given models $(H_1,T_1),(H_2,T_2)$ of $ \Sigma_{K,m}$, we can write 
   $$(H_1,T_1)=(H_{\fin},T_{\fin}) \oplus [(H_1,T_1)\ominus (H_{\fin},T_{\fin})]$$ and $$(H_2,T_2)=(H_{\fin},T_{\fin}) \oplus [(H_2,T_2)\ominus (H_{\fin},T_{\fin})].$$ Set $(H_3,T_3)=(H_{\fin},T_{\fin})\oplus
   [(H_1,T_1)\ominus (H_{\fin},T_{\fin})] \oplus [((H_2,T_2)\ominus (H_{\fin},T_{\fin})]$. 
   Clearly $\sigma(T_3)=K$. Since the eigenspaces associated to isolated eigenvalues of finite multiplicity are contained in $H_{\fin}$,  we have $(H_3,T_3)\models \Sigma_{K,m}$.  It is clear that $(H_i,T_i)$ embeds into $(H_3,T_3)$ for $i=1,2$.

Given three models $(H_0,T_0),(H_1,T_1),(H_2,T_2)$ of $\Sigma_{K,m}$ with $(H_0,T_0)\subseteq (H_i,T_i)$ for $i=1,2$, we can amalgamate the structures 
$(H_1,T_1)$ and $(H_2,T_2)$ over $(H_0,T_0)$ by repeating the construction above over $(H_0,T_0)$. In this case, we can define
   $(H_3,T_3)=(H_0,T_0)\oplus
   [(H_1,T_1)\ominus (H_0,T_0)] \oplus [((H_2,T_2)\ominus (H_0,T_0)]$. 
   As in the case of the joint embedding property, we have $(H_3,T_3)\models \Sigma_{K,m}$ and is an amalgalm of $(H_1,T_2)$ and $(H_2,T_2)$ over $(H_0,T_0)$.
\end{proof}

\begin{rem}
If the reader is familiar with the characterization of forking from Section \ref{sec:stability}, the construction in the previous argument witnessing the amalgamation property is the free amalgam, in the sense of forking independence, of the structures $(H_1,T_1)$ and $(H_2,T_2)$ over $(H_0,T_0)$.
\end{rem}

\section{Quantifier elimination}\label{sec:QE}

In this section, we show that, after adding a symbol to name the adjoint, each theory $\Sigma_K$ admits quantifier-elimination.

Let $L^*$ be the language obtained by adding to $L$ a new unary function symbol $T^*$, which also has modulus of uniform continuity the identity function.  Let $\Sigma_0^*$ be the $L^*$-theory extending $\Sigma_0$ which states that $T^*$ is a linear operator with $\|T^*\|\leq 1$ that is the adjoint of $T$:  this can be expressed by the single universal axiom $\sup_{x,y}|\langle Tx,y\rangle-\langle x,T^*y\rangle|=0$.  Note now that if $(H_1,T_1,T_1^*),(H_2,T_2,T_2^*)$ are models of $\Sigma_0^*$ with $(H_1,T_1,T_1^*)\subseteq (H_2,T_2,T_2^*)$, then $H_1$ is reducing for $T_2$.

We let $\Sigma_{\operatorname{normal}}^*$, $\Sigma_K^*$, and $\Sigma_{K,m}^*$ denote the obvious extensions of the corresponding $L$-theories.

\begin{thm}\label{QE}
For any nonempty compact $K\subseteq \mathbb D$, the theory $\Sigma_{K}^*$ has quantifier-elimination.
\end{thm}

\begin{proof}
Fix models $(H_1,T_1,T_1^*),(H_2,T_2,T_2^*)\models \Sigma_{K}^*$ with $(H_2,T_2,T_2^*)$ $\kappa^+$-saturated, where $\kappa$ is the density character of $H_1$.  Further suppose that $(H_0,T_0,T_0^*)$ is a substructure of $(H_1,T_1,T_1^*)$ and $f:(H_0,T_0,T_0^*)\to (H_2,T_2,T_2^*)$ is an embedding.  It suffices to show that $f$ can be extended to a substructure of $(H_1,T_1,T_1^*)$ whose domain properly contains $H_0$.  For this purpose, we may assume, without loss of generality, that $H_{T_1,\operatorname{fin}}\subseteq H_0$.  It suffices then to show that, given $x\in H_1\ominus H_0$, we may find an extension of $f$ to an embedding from a substructure of $(H_1,T_1,T_1^*)$ containing $x$ in its domain to $(H_2,T_2,T_2^*)$. Since $H_0$ is reducing for $T_0^*$, we have that $p(T,T^*)(x)\in H_1\ominus H_0$ for all $p\in \mathbb C[X,Y]$.  Consequently, it suffices to find $y\in H_2\ominus f(H_0)$ such that $\langle p(T_1,T_1^*)(x),x\rangle=\langle p(T_2,T_2^*)(y),y\rangle$ for all $p\in \mathbb C[X,Y]$. 
Fix finitely such polynomials $p_1,\ldots,p_m$ and $\epsilon>0$; setting $\alpha_i:=\langle p_i(T_1,T_1^*)(x),x\rangle$, by saturation, it suffices to find $y\in H_2\ominus f(H_0)$ such that $|\langle p_i(T_2,T_2^*)(y),y\rangle-\alpha_i|<\epsilon$ for all $i=1,\ldots,m$.

Set $\mu$ to be the scalar measure on the spectrum corresponding to $x$, that is, $\mu(I)=\langle E_1(I)x,x\rangle$ for every Borel set $I\subseteq K$; here $E_1$ is the spectral measure associated to $T_1$.    Choose a sufficiently fine grid partition of the plane, by which we mean a partition of the plane into small squares, where, say, the left and bottom edge of each square is included in the set, while the top and right edge are not included.  Let $I_1,\ldots,I_t$ be the elements of the partition that intersect $K$.  For each $k=1,\ldots,t$, we define an element $y_k\in E_2(I_k)$ (where $E_2$ is the spectral projection associated to $T_2$) as follows.  If $E_2(I_k\cap K)$ has finite rank, then set $y_k:=0$.  If $E_2(I_k\cap K)$ has infinite rank, then, by saturation, its dimension is greater than the dimension of $f(H_0)$, whence we may find a unit vector $y_k\in E_2(I_k)\cap (H_2\ominus f(H_0))$.  Set $y:=\sum_{k=1}^t \mu(I_k)y_k$, which is an element of $H_2\ominus f(H_0)$.  If $E_2(I_k)$ has finite rank, then $E_1(I_k)\subseteq H_{T_1,\operatorname{fin}}$; since $x\in H_1\ominus H_{T_1,\operatorname{fin}}$, we have $\langle E_2(I_k)y,y\rangle=0=\langle E_1(I_k)x,x\rangle$.  If $E_2(I_k)$ has infinite rank, we have
$$\langle E_2(I_k)y,y\rangle=\langle \mu(I_k)y_k,y_k\rangle=\mu(I_k)=\langle E_1(I_k)x,x\rangle.$$ For each $k=1,\ldots,t$, let $\lambda_k$ be an element of $I_k\cap K$.  We then have 
$$\langle \sum_{k=1}^t p_i(\lambda_k)E_2(I_k)y,y\rangle=\langle \sum_{k=1}^t p_i(\lambda_k)E_1(I_k)x,x\rangle.$$
On the other hand, if the grid is sufficiently fine so that, for all $i=1,\ldots,m$, $|p_i(a)-p_i(b)|<\epsilon$ whenever $a$ and $b$ belong to the same $I_k$, then by Fact \ref{fact: measures}, we have $\|p_i(T_j,T_j^*)-\sum_{k=1}^t p_i(\lambda_k)E_j(I_k)\|<\epsilon/2$ for $j=1,2$.  Consequently, this $y$ is as desired.
\end{proof}

\begin{cor}\label{modelcomplete}
Every theory $\Sigma_K$ is model-complete.
\end{cor}

\begin{proof}
    As already mentioned above, an inclusion $(H_1,T_1)\subseteq (H_2,T_2)$ between models of $\Sigma_K$ must be such that $H_1$ is $T_2$-reducing and $T_1^*=T_2^*\upharpoonright\! H_1$.  Consequently, we have that $(H_1,T_1,T_1^*)\subseteq (H_2,T_2,T_2^*)$ as $L^*$-structures.  Since both of these structures are models of $\Sigma_K^*$, Theorem \ref{QE} implies that this inclusion is elementary as $L^*$-structures, whence the original inclusion $(H_1,T_1)\subseteq (H_2,T_2)$ is elementary. 
\end{proof}

\begin{ques}
Do the theories $\Sigma_K$ have quantifier-elimination?
\end{ques}

Note that the previous question has a positive answer in some cases:

\begin{thm}
If $K\subseteq [-1,1]$ (corresponding to the case of a self-adjoint operator) or $K\subseteq \mathbb S^1$ (corresponding to the case of a unitary operator), then $\Sigma_K$ admits quantifier-elimination.
\end{thm}

\begin{proof}
We use the notation from the proof of Theorem \ref{QE}.  The only place in the proof of Theorem \ref{QE} where the symbol for the adjoint was used was the conclusion that the subspace $H_0$ was reducing for $T_1$.  If $T_1$ is self-adjoint, then $H_0$ is obviously reducing for $T_1$.  If $T_1$ is unitary, then it suffices to show that one may extend the embedding $f:(H_0,T_0)\to (H_2,T_2)$ to the smallest closed subspace of $H_1$ containing $H_0$ and closed under $T_1$ and $T_1^*$, for then one may assume that $H_0$ is actually reducing for $T_1$ and the remainder of the proof of Theorem \ref{QE} remains the same.  To verify this claim, first note that every element in the smallest subspace $H$ of $H_1$ containing $H_0$ and closed under $T_1$ and $T_1^*$ is of the form $x:=\sum_{i=0}^n T_1^{-i}b_i$, with $b_0,\ldots,b_n\in H_0$.  We claim that we may extend $f$ to a function $f':H\to H_2$ by defining $f'(x):=\sum_{i=0}^n T_2^{-i}f(b_i)$ for $x$ as in the previous sentence.  To see that this is well-defined, suppose that $x=0$.  Then $\sum_{i=0}^n T_1^{n-i}b_i=0$, whence $\sum_{i=0}^n T_2^{n-i}f(b_i)=0$ since $f$ intertwines $T_1$ and $T_2$.  Applying $T_2^{-n}$ to both sides yields $\sum_{i=0}^n T_2^{-i}f(b_i)=0$, showing that the extended function $f'$ is well-defined.  Reversing these steps and using that $f$ is injective shows that $f'$ is also injective.  Finally, note that the extended function $f'$ is linear and preserves the inner product, whence $f'$ extends to an embedding of $\bar H$ into $H_2$; moreover, it is clear that this extension of $f'$ to the closed space $\bar H$ still intertwines $T_1$ and $T_2$, establishing the claim and completing the proof.
\end{proof}

Returning to the general case of arbitrary normal operators, we have the following consequence of Theorem \ref{QE} and Corollary \ref{modelcomplete}:

\begin{cor}
$\Sigma_{\operatorname{normal}}$ has a model companion, namely $\Sigma_{\mathbb D}$.  $\Sigma_{\operatorname{normal}}^*$ has a model completion, namely $\Sigma_{\mathbb D}^*$.
\end{cor}

\begin{proof}
It suffices to note that every model of $\Sigma_{\operatorname{normal}}$ embeds into a model of $\Sigma_{\mathbb D}$, which follows from the fact that, for any two models $(H_1,T_1),(H_2,T_2)$ of $\Sigma_{\operatorname{normal}}$, one has $\sigma(T_1\oplus T_2)=\sigma(T_1)\cup \sigma(T_2)$.
\end{proof}

\begin{cor}\label{modcomp}
For any compact subset $K\subseteq \mathbb D$, $\Sigma_K$ has a model companion, namely $\Sigma_{K,m_\infty}$, where $m_\infty(\lambda)=\infty$ for all $\lambda \in \isol(K)$, and $\Sigma_K^*$ has a model completion, namely $\Sigma_{K,m_\infty}^*$.
\end{cor}

\begin{proof}
Once again, it suffices to see that any model of $\Sigma_K$ embeds in a model of $\Sigma_{K,m_\infty}$.  This follows by simply adding a model of $\Sigma_{K,m_\infty}$ as a direct summand.
\end{proof}

\section{Types and separable categoricity}\label{types-categ}

Quantifier elimination also allows us to completely describe types.  Until further notice, we fix a completion $\Sigma=\Sigma_{K,m}$ of $\Sigma_{\operatorname{normal}}$.  Given $(H,T)\models \Sigma$ and $B\subseteq H$, let $\langle B\rangle_0$ denote the closed subspace generated by all elements of $H$ of the form $p(T,T^*)b$, where $p(X,Y)\in \mathbb C[X,Y]$ and $b\in B$.

\begin{prop}\label{typedescription}
Suppose that $(H,T)\models \Sigma$, $B\subseteq H$ is a parameterset, and $\vec a=(a_1,\ldots,a_n),\vec a'=(a_1',\ldots,a_n')\in H^n$ are tuples.  Let $P_B$ denote the orthogonal projection onto $\langle B\rangle_0$.  Then $\vec a$ and $\vec a'$ have the same type in $(H,T)$ over $B$ if and only if:
\begin{enumerate}
    \item $P_B(a_i)=P_B(a_i')$ for all $i=1,\ldots,n$.
    \item Setting $c_i=a_i-P_B(a_i)$ and $c_i'=a_i'-P_B(a_i')$, we have $$\langle p(T,T^*)c_i,c_j\rangle=\langle p(T,T^*)c_i',c_j'\rangle$$ for all polynomials $p(X,Y)\in \mathbb C[X,Y]$ and all $i,j=1,\ldots,n$.  
\end{enumerate}
\end{prop}

\begin{proof}
Since $\Sigma^*$ is a definitional expansion of $\Sigma$, we may as well assume that we work in the language $L^*$ where the theory of $(H,T,T^*)$ has quantifier elimination.  If $\vec a$ and $\vec a'$ have the same type in $(H,T)$ over $B$, then they also have the same type over $\dcl(B)$; since $\langle B\rangle_0\subseteq \dcl(B)$, they have the same type over $\langle B\rangle_0$, whence (1) follows.  In an elementary extension of $(H,T)$, $\vec a$ and $\vec a'$ are conjugate by an automorphism fixing $B$ and thus $\langle B\rangle_0$; this automorphism must then map $\vec c$ to $\vec c'$ by (1), whence (2) follows.

Conversely, assume that (1) and (2) holds.  By quantifier elimination, it suffices to check that $\langle p(T,T^*)a_i,a_j\rangle=\langle p(T,T^*)a_i',a_j'\rangle$ for all $p(X,Y)\in \mathbb C[X,Y]$ and all $i,j=1,\ldots,n$.  The desired equality follows from (1) and (2) by noting that
$$\langle p(T,T^*)a_i,a_j\rangle=\langle p(T,T^*)P_B(a_i),P_B(a_j)\rangle+\langle p(T,T^*)c_i,c_j\rangle$$ and the analogous equation obtained by replacing $a_i$ and $a_j$ with $a_i'$ and $a_j'$.
\end{proof}

As an application of our newfound understanding of types, we characterize the algebraic and definable closure in models of $\Sigma_{\operatorname{normal}}$.

\begin{defn}
For a model $(H,T)$ of $\Sigma_{\operatorname{normal}}$ and $B\subseteq H$, we write $\langle B\rangle$ for the closed subspace of $H$ generated by $H_{\fin,T}$ and all vectors of the form $p(T,T^*)b$, with $p(X,Y)\in \mathbb C[X,Y]$ and $b\in B$.
\end{defn}

Notice that, by Corollary \ref{cor:Hfin-invariant} and by construction, the space $\langle B \rangle $ is a closed subspace of $H$ which is reducing for $T$ that contains $B$ and all finite dimensional eigenspaces corresponding to isolated points of the spectrum. Furthermore, by the arguments behind the proof of Lemma \ref{definablefunctionalcalculus} and Lemma \ref{lem:Application-definability}, both closed spaces $\langle B \rangle$, $\langle B \rangle_0$
are closed under projections of the form $E(\Sigma_1)$, where $\Sigma_1$ is $\epsilon$-separated from the rest of the spectrum for soe $\epsilon>0$.

\begin{prop}\label{CharAlgCls}
For $(H,T)\models \Sigma_{\operatorname{normal}}$ and $B\subseteq H$, we have $\acl(B)=\langle B \rangle$.
\end{prop}

\begin{proof}
First assume that $x\in \langle B \rangle$. If $x\in H_{\fin,T}$, then $x\in \acl(\emptyset)\subseteq \acl(B)$.  If $x=p(T,T^*)b$, then $x\in \dcl(B)$ since the graph of $T^*$ is definable.  Since $\acl(B)$ is always a closed subspace, we see that $\langle B\rangle \subseteq \acl(B)$.


Assume now that $x\not \in \langle B \rangle$. Write $x=x_B+x_1$, where $x_B=P_{\langle B\rangle}(x)$; note that $x_1\neq 0$. It suffices to show that, in an elementary extension of $(H,T)$, there are infinitely many realizations of $\tp(x_1/B)$ separated by a uniform distance $\delta>0$.
Write $H=\langle B\rangle \oplus H_1$, where $H_1=\langle B\rangle^\perp$. Since $\langle B\rangle$ is reducing for $T$, so is $H_1$. Let $T_1=T\upharpoonright_{H_1}$.
Set $(H_\omega,T_\omega):=\bigoplus_{i<\omega}(H_1,T_1)$ and set $(H',T')=(\langle B\rangle\oplus H_\omega,T\upharpoonright\!\langle B\rangle\oplus T_\omega)$. Since $\sigma(T_1)\subseteq \sigma(T)$, we have $\sigma(T')\subseteq \sigma(T)$; since we can identify $(H,T)$ with a substructure of $(H',T')$, we also have $\sigma(T)\subseteq \sigma(T')$ and thus $\sigma(T)=\sigma(T')$. Furthemore, the eigenspaces associated to the isolated points in $\sigma(T)$ belong to $\langle B\rangle$, whence they have the same dimension in both spaces.  It follows that $(H,T)\equiv (H',T')$. Moreover, we get using quantifier elimination that $(H,T)\preceq (H',T')$. Let $x_i$ be the copy of $x_1$ in the $i$-th component of the sum $\bigoplus_{i<\omega}(H_1,T_1)$. Then $\tp(x_B+x_i/B)=\tp(x_B+x_1/B)$ and $d(x_b+x_i,x_B+x_j)=\sqrt{2}\|x_1\|$ whenever $i\neq j$. Thus the
family $\{x_B+x_i:i<\omega\}$ is uniformly separated and the proof is complete.
\end{proof}

\begin{prop}
For $(H,T)\models\Sigma_{\operatorname{normal}}$ and $B\subseteq H$, we have that $\operatorname{dcl}(B)=\langle B\rangle_0$.
\end{prop}
\begin{proof}
    It is enough to prove that $\text{dcl}(B)\subseteq \langle B\rangle_0$. Suppose $v\not \in \langle B\rangle_0$; we show $v\notin \operatorname{dcl}(B)$.  Write $v=P_{B}(v)+w$ (recall that $P_{B}(v)$ stands for the projection onto the closed space $\langle B\rangle_0$) and note that $w\not=0$.  It suffices to show that $\tp(P_B(c)+w/B)=\tp(P_B(c)-w/B)$,
    as from this it follows that $v\not \in \text{dcl}(B)$.  However, this follows immediately from Proposition \ref{typedescription} by noting that $\langle p(T,T^*)w,w\rangle=\langle p(T,T^*)(-w),-w\rangle$ for all polynomials $p(X,Y)\in \mathbb C[X,Y]$.
\end{proof}

Since $*$-polynomials are dense in the space of all continuous functions on $K$, item (2) in Proposition \ref{typedescription} can be rephrased as $\mu_{c_i,c_j}=\mu_{c_i',c_j'}$ for all $i,j=1,\ldots,n$.  Consequently, restricting to the case $B=\emptyset$, we see that there is a well-defined map $\Phi_n:S_n(\Sigma)\to M(K)^{n^2}$ defined by $\Phi(p)(i,j)=\mu_{c_i,c_j}$, where $i,j=1,\ldots,n$ and $\vec c=(c_1,\ldots,c_n)$ is any realization of $p$.  In what follows, we view $M(K)$ as equipped with its weak*-topology and $M(K)^{n^2}$ with the corresponding product topology.

\begin{prop}\label{homeo}
For each $n$, $\Phi_n$ is a homeomorphism of $S_n(\Sigma)$ onto its image.  
\end{prop}

\begin{proof}
Fix a net $(p_i)$ from $S_n(\Sigma)$; we must show that $p_i\to p$ in the logic topology if and only if $\Phi(p_i)\to \Phi(p)$ in $M(K)^{n^2}$.
To see this, first suppose that $p_i\to p$ in the logic topology and fix realizations $\vec c^i$ and $\vec c$ of $p_i$ and $p$ respectively; we must show that $\int fd\mu_{c_j^i,c_k^i}\to \int fd\mu_{c_j,c_j}$ for all $j,k=1,\ldots,n$ and all continuous functions $f:K\to \mathbb{C}$.  However, since $*$-polynomials are dense in the space of all continuous functions on $K$ and $\int p(z,\bar z)d\mu_{c_j^i,c_k^i}=\langle p(T,T^*)c_j^i,c_k^i\rangle$ and $\int p(z,\bar z)d\mu_{c_j,c_k}=\langle p(T,T^*)c_j,c_k\rangle$, the result follows by our assumption that $p_i \to p$ in the logic topology.  The backwards direction follows from quantifier-elimination by reversing the reasoning in the first part of the proof.
\end{proof}

In the case of $1$-types, we have that $\Phi_1$ takes values in the space $M(K)_+$ consisting of the positive measures in $M(K)$.  In fact, we now show that $M(K)_+$ is precisely the image of $\Phi_1$.  First, we need to recall a basic measure theory fact, which follows from the Pormanteau theorem characterizing weak*-convergence (see, for example, \cite[Theorem 17.20]{kechris}); for convenience, we state it only for subsets of the complex plane.

\begin{fact}\label{portmanteau}
Suppose that $\mathcal{E}$ is a collection of Borel subsets of $\mathbb{C}$ with the finite intersection property and such that any open subset of $\mathbb{C}$ can be written as a countable union of elements from $\mathcal{E}$.  Then for any $\mu_n,\mu\in M(\mathbb{C})$, we have that $\mu_n\to \mu$ in the weak*-topology if and only if $\mu_n(B)\to \mu(B)$ for each $B\in \mathcal{E}$.
\end{fact}

We will apply the previous fact with the set $\mathcal{E}$ of ``half-open dyadic squares'' in $\mathbb{C}$ whose elements are of the form $$B_{k,n}:=\{(x,y)\in \mathbb{C} \ : \ k/2^n\leq x<(k+1)/2^n, \ k/2^n\leq y<(k+1)/2^n\},$$ for $n\in \mathbb N$ and $k\in \mathbb Z$.  Note that, for fixed $n\in \mathbb N$, the $E_{k,n}$'s are pairwise disjoint.

\begin{prop}\label{surjective}
The map $\Phi_1:S_1(\Sigma)\to M(K)_+$ is surjective.
\end{prop}

\begin{proof}
Fix $\mu\in M(K)_+$ and $n\in \mathbb{N}$.  Fix also $(H,T)\models \Sigma$.  Fix $n\geq 1$ and let $B_1,\ldots,B_{t(n)}$ denote those $B_{k,n}$'s that intersect $K$.  For each $i=1,\ldots,t(n)$, let $x^n_i\in E(I_i)(H)$ be such that $\|x_i^n\|=\sqrt{\mu(I_i\cap K)}$.  Setting $x_n:=\sum_{i=1}^{t(n)} x_i^n$ and $p_n:=\operatorname{tp}(x_n)$, we have that $\mu_{p_m}(I_i)=\mu_p(I)$ for all $i=1,\ldots,t(n)$ and $m\geq n$.  Fact \ref{portmanteau} implies that $\mu_{p_n}\to \mu$ in the weak*-topology.
Supposing, without loss of generality, that $\mu$ is a probability measure, we have that $\|x_n\|=1$ for all $n$.  Using Proposition \ref{homeo} and the fact that the set of types of unit vectors in $S_1(\Sigma)$ is compact, we have that the image under $\Phi$ of such types is weak*-closed, whence $\mu$ is in the image of $\Sigma$, as desired.
\end{proof}

\begin{ques}
For $n>1$, can one characterize the image of $\Phi_n$?
\end{ques}

\begin{rem}
    Given the above identification of $S_1(\Sigma)$ and $M(K)_+$, our proof of quantifier elimination (Theorem \ref{QE}) is reminiscent of the Krein-Milman theorem (presuming one chooses the $y_k$'s in that proof to be eigenvectors, as the associated  measures will then be extreme measures). 
\end{rem}

Having described the logic topology on the space of types, we next describe the metric on the space of types.  For simplicity, we will only consider the case of $1$-types.  For ease of notation, for $p\in S_1(\Sigma_{K,m})$, we set $\mu_p:=\Phi_1(p)$.  Note that, for each Borel set $A\subseteq K$, one has $\sqrt{\mu_p(A)}=\|E(A)x\|$, where $x$ is a realization of $p$ in some model $(H,T)$ of $\Sigma_{K,m}$ and $E$ is the spectral measure associated to $T$. 

\begin{prop}\label{distanceontypes}
Fix $p,q\in S_1(\Sigma_{K,m})$ and realizations $c$ of $p$ and $d$ of $q$ in a common model $(H,T)$ of $\Sigma_{K,m}$.  Then
$$d(p,q)^2=\sup\{\sum_n \left|\|E(A_n)c\|-\|E(A_n)d\|\right|^2 \ : (A_n) \text{ a measurable partition of }K\}.$$
\end{prop}

\begin{proof}
Let $r$ denote the right hand side of the above display.  We first show that $d(p,q)^2\geq r$.  Fix arbitrary realizations $a\models p$ and $b\models q$ in some common model $(H,T)$ of $\Sigma_{K,m}$ such that $d(p,q)=\|a-b\|$.  For any measurable partition $(A_n)$ of $K$, we have
$$d(p,q)^2= \|a-b\|^2=\sum_n \|E(A_n)a-E(A_n)b\|^2\geq \sum_n \left| \|E(A_n)a\|-\|E(A_n)b\|\right|^2.$$  By taking the supremum over all partitions, we wee that $d(p,q)^2\geq r$.

We now show that $d(p,q)^2\leq r$.  Fix a measurable partition $(A_n)$ of $K$.  For each $n$, let $x_n,y_n\in E(A_n)(H)$ be parallel vectors satisfying $\|x_n\|=\|E(A_n)c\|$ and $\|y_n\|=\|E(A_n)d\|$.  Then $\|x_n-y_n\|=\left| \|E(A_n)c\|-\|E(A_n)d\|\right|$.  Setting $x:=\sum_n x_n$, $y=\sum_n y_n$, and $p_n:=\operatorname{tp}(x_n)$ and $q_n=\operatorname{tp}(y_n)$, we have 
$$d(p_n,q_n)^2\leq \|x-y\|^2=\sum_n \|x_n-y_n\|^2=\sum_n \left| \|E(A_n)(c)\|-\|E(A_n)(d)\|\right|^2\leq r.$$  By taking finer square partitions as in the proof of Proposition \ref{surjective} (and the discussion preceding it), we have that $p_n\to p$ and $q_n\to q$ in the logic topology.  Since the metric is lower semicontinuous in the logic topology, we conclude that $d(p,q)^2\leq r$, as desired.
\end{proof}

\begin{rem}
The formula in Proposition \ref{distanceontypes} can be written as 
$$d(p,q)^2=\sup\{\sum_n |\sqrt{\mu_p(A_n)}-\sqrt{\mu_q(A_n)}|^2  : (A_n) \text{ a measurable partition of }K\}.$$  This is reminiscient of the formula for the total variation distance which reads
$$\|\mu_p-\mu_q\|=\sup\{\sum_n |\mu_p(A_n)-\mu_q(A_n)|  : (A_n) \text{ a measurable partition of }K\}.$$
In fact, defining the $d$-topology on $M(K)_+$ to be that induced by the metric on $S_1(\Sigma_{K,m})$, we have that the norm topology on $M(K)_+$ is finer than the $d$-topology. Indeed, this follows from the fact that, for all $r,s\geq 0$, we have that $|\sqrt r-\sqrt s|^2\leq |r-s|$ and thus $d(p,q)^2\leq\| \mu_p-\mu_q\|.$

\end{rem}

\begin{ques}
Does the norm topology on $M(K)_+$ coincide with the $d$-topology
\end{ques}

In regards to the previous question, it is interesting to point out that the following are equivalent for a given theory $\Sigma_{K,m}$:
\begin{enumerate}
    \item The norm topology and weak*-topology on $M(K)$ coincide.
    \item The $d$-topology and weak*-topology on $M(K)$ coincide.
    \item $K$ is finite.
\end{enumerate}

The equivalence of (1) and (3) is standard fare and the equivalence of (2) and (3) will follow from Proposition \ref{homeo} and the fact, to be proved below in in Corollary \ref{omega-cat}, that $\Sigma_{K,m}$ is separably categorical if and only if $K$ is finite. 

\begin{rem}
   In the case that $K$ is countable, the formula appearing in Proposition \ref{distanceontypes} takes on the following simpler form:
   $$d(p,q)^2=\sum_{\lambda \in K}\Big{|} \|P_{H_\lambda}(c)\|- \|P_{H_\lambda}(d)\|\Big{|}^2.$$
\end{rem}

The formula in Proposition \ref{distanceontypes} can also be modified to handle 1-types over parameters.  Given $(H,T)\models \Sigma_{K,m}$ and $B\subseteq H$, Proposition \ref{typedescription} shows that every type $p\in S_1(B)$ yields a uniquely determined type $p_B\in S_1(T)$:  if $a$ is a realization of $p$, then $p_B$ is the type of $a-P_{\langle B\rangle_0}(a)$.  Moreover, if $c\in H$ is any realization of $p_B$, then $P_B(a)+c$ is a realization of $p$.  These remarks immediately imply:

\begin{cor}
Let $(H,T)\models \Sigma_{K,m}$ be sufficiently saturated and let $B\subseteq H$ be such that $B=\langle B\rangle_0$.
For $p,q\in S_1(B)$ and $a,b\in H$ realizing $p$ and $q$ respectively, we have
$$d(p,q)^2=\|P_{B}(a)-P_{B}(b)\|^2+ d(p_B,q_B)^2.$$
\end{cor}


We now characterize the principal types. We first start with $1$-types.

\begin{prop}
Consider a type
$p(x)\in S_1(\Sigma_{K,m})$ such that the formula $\|x\|=1$ belongs to $p(x)$. Then $p(x)$ is principal if and only if $\mu_p(\isol(K))=1$.
\end{prop}

\begin{proof}
Since $\|x\|=1$, $\mu_p$ defines a probability measure on the Borel subsets of $K$.
Assume first that $\mu_p(\isol(K))=1$; it suffices to show that $p$ is realized in all models of $\Sigma_{K,m}$.  Take $(H,T)\models \Sigma_{K,m}$; for each $\lambda\in K$, write $H_{\lambda}$ for the eigenspace associated to $\lambda$ in the structure $(H,T)$. For each $\lambda \in \isol(K)$, $\dim(H_{\lambda})=m(\lambda)$, so all eigenspaces associated to isolated eigenvalues have positive dimension. Since $\mu_p(\isol(K))=1$ and $\isol(K)$ is countable, there is a sequence $(\lambda_n)$ from $\isol(K)$ such that $\mu_p=\sum \alpha_n \delta_{\lambda_n}$, where $\delta_{\lambda_n}$ is the probability measure concentrated on the singleton $\{\lambda_n\}$, $\alpha_n\geq 0$, and $\sum_n \alpha_n=1$. For each $n$, choose $v_n\in H_{\lambda_n}$ with $\|v_n\|=1$ and set $w=\sum_{n} \alpha_n v_n$. Then $w\in H$ is a realization of $p$ since $\langle E(\{\lambda_n\})w,w\rangle=\langle \alpha_nv_n,v_n\rangle=\alpha_n=\mu_p(\{\lambda_n\})$.

    For the other direction, suppose $\mu_p(\isol(K))<1$, whence $\mu_p(K')>0$, where $K':=K\setminus \isol(K)$ is the set of accumulation points of $K$; it suffices to find a model of $\Sigma_{K,m}$ which omits $p$.  Let $Q$ be a dense Borel subset of $K'$ for which $\mu_p(K'\setminus Q)>0$.  For each $\lambda \in K$, let $(H_{\lambda},T_{\lambda})$
    be the structure satisfying $\dim(H_\lambda)=m(\lambda)$ and $T_\lambda(v)=\lambda v$ for every $v\in H_\lambda$ (recalling that $m(\lambda)=\infty$ when $\lambda\in K'$).  Let $(H,T):=\bigoplus_{\lambda \in \isol(K)\cup Q}(H_\lambda,T_\lambda)$.  Then $(H,T)$ is a model of $\Sigma_{K,m}$ which omits $p$.    
\end{proof}

Note that the previous proposition holds for an arbitrary $1$-type by rescaling, that is, $p(x)$ is principal if and only if $\mu_p(\isol(K))=r$, where $r\in [0,1]$ is the unique number such that $\|x\|=r$ belongs to $p$.  

The above analysis can be extended to handle $n$-types. We first need a simple observation:

\begin{lem}\label{rem:abs-cont}
     Fix $v,w\in H$ and consider the measures $\mu_v$ and $\mu_{v,w}$ on the Borel subsets of $\mathbb{D}$. Then $\mu_{v,w}$ is absolutely continuous with respect to $\mu_{v}$.
\end{lem}

\begin{proof}
Assume that $\mu_v(A)=0$, so $\langle E(A)v,v\rangle=0$. Since $E(A)$ is a projection, it satisfies $E(A)^2=E(A)$ and $E(A)^*=E(A)$. Then we have $0=\langle E(A)v,v\rangle=\langle E(A)(E(A)v),v\rangle=\langle E(A)v,E^*(A)v\rangle=\langle E(A)v,E(A)v\rangle$, whence $\|E(A)v\|=0$. It follows that $\langle E(A)v,w\rangle=0$.
\end{proof}

\begin{cor}\label{isolatedntypes}
    
Consider a type
$p(\vec x)=p(x_1,\dots,x_n)\in S_n(\Sigma_{K,m})$ and let $p_i(x_i)$ be the restriction of $p$ to the variable $x_i$.  Then $p(\vec x)$ is principal if and only if each $p_i(x_i)$ is principal.
\end{cor}

\begin{proof}
    Clearly, if $p(\vec x)$ is principal, so is $p_i(x_i)$. Conversely, assume that for each $i=1,\ldots,n$, we have $\mu_{p_i}(\isol(K))=r_i^2$, where $\|x_i\|=r_i$ belongs to $p_i$; we show that $p$ is realized in all models of $\Sigma_{K,m}$. Note that for any Borel set $A\subseteq K$ with $A\cap \isol(K)=\emptyset$, we have $\mu_{p_i}(A)=0$; 
    by Lemma \ref{rem:abs-cont} the measure $\mu_{p_i,p_j}$ is absolutely continuous with respect to $\mu_{p_i}$, so
    we also have $\mu_{p_i,p_j}(A)=0$. Take $(H,T)\models \Sigma_{K,m}$ and a saturated elementary extension $(\hat H,\hat T)\succeq (H,T)$. 
    Take a realization $(v_1,\dots,v_n)\in \hat H $ of $p(\vec x)$. Since the measures $\mu_{p_i,p_j}$ are supported in $\isol(K)$, $p(\vec x)$ is determined by $\{\mu_{p_i,p_j}(\{\lambda\})=\langle P_{\lambda}(v_i),P_{\lambda}(v_j))\rangle: \lambda\in \isol(K)\}$ together with the conditions $\|x_i\|=r_i$ . Since for each $\lambda\in \isol(K)$ we have $m(\hat H_{\lambda})=m( H_{\lambda})$, we can find 
    $(w_1,\dots,w_n)\in H$ with $\langle P_{\lambda}(w_i),P_{\lambda}(w_j))\rangle=\langle P_{\lambda}(v_i),P_{\lambda}(v_j))\rangle$ and $\|w_i\|=r_i$ for all $i,j=1,\ldots,n$ and $\lambda\in \isol(K)$. It follows that $(w_1,\dots,w_n)\in H$ is a realization of $p(\vec x)$.
\end{proof}

\begin{cor}\label{omega-cat}
The theory $\Sigma_{K,m}$ is $\omega$-categorical if and only if $K$ is finite.
\end{cor}

\begin{proof}
    The theory $\Sigma_{K,m}$ is $\omega$-categorical if and only if every type $p(\vec x)\in S_n(\Sigma_{K,m})$ is principal; by Corollary \ref{isolatedntypes}, the latter condition is equivalent to the statement that  $K$ only has isolated points, which, since $K$ is compact, is equivalent to the condition that $K$ is finite.
\end{proof}

We offer a second proof of Corollary \ref{omega-cat} that is perhaps more elementary:



\begin{proof}[Alternative proof of Corollary \ref{omega-cat}]
    Assume first that $K$ is finite.  The separable model of $\Sigma_{K,m}$ is necessarily then isomorphic to $\bigoplus_{\lambda\in K}(H_\lambda,T_\lambda)$, where $\dim(H_\lambda)=m(\lambda)$ and $T_\lambda x=\lambda x$ for all $x\in H_\lambda$.  

     Assume now that $K$ is infinite and take a separable model $(H,T)$ of $\Sigma_{K,m}$.
     Since $K$ is compact and infinite, it has an accumulation point $\lambda_0$. Consider $H_{\lambda_0}$, the corresponding eigenspace in $(H,T)$. If $\dim(H_{\lambda_0})>0$, then the structure $(H',T')=(H,T)\ominus (H_0,T\upharpoonright_{H_0})$ is not isomorphic to $(H,T)$ (the eigenspaces corresponding to $\lambda_0$ have different dimension) but $(H',T')$ is still a model of $\Sigma_{K,m}$.  If $\dim(H_{\lambda_0})=0$, then by Upward L\"owenheim-Skolem, we can choose $(H,T)\preceq (H',T')$ separable with $\dim(H'_{\lambda_0})>0$; once again, $(H,T)$ and $(H',T')$ are not isomorphic since the eigenspaces corresponding to $\lambda_0$ have different dimension. 
\end{proof}

\begin{rem}
If $K$ is finite and $m(K)\subseteq \mathbb{N}$, then in fact $\Sigma_{K,m}$ has a unique model up to isomorphism.
\end{rem}

If we allow perturbations, we always have separable categoricity of $\Sigma_{K,m}$; in fact, that is precisely the model-theoretic reformulation of Fact \ref{fact:WNB}.  For more on perturbations in this context, see Definition \ref{def:pert} below.

\begin{prop}
The theory $\Sigma_{K,m}$ is $\omega$-categorical up to perturbations.
\end{prop}


\begin{rem}
In general, the theory $\Sigma_{K,m}$ is not $\aleph_1$-categorical up to perturbations. For example, consider the case $K=\{0,1\}$ and $m(0)=m(1)=\infty$. There are three models of $\Sigma_{K,m}$ up to isomorphism of density character $\aleph_1$, namely: the ``rich'' one where $\dim(H_0)=\dim(H_1)=\aleph_1$, and two ``poor'' ones, the first one with $\dim(H_0)=\aleph_0$ and $\dim(H_1)=\aleph_1$, and the second one satisfying $\dim(H_0)=\aleph_1$ and $\dim(H_1)=\aleph_0$. This picture will not change if we consider the models up to perturbation. As long as we have at least two points in $K$ that are isolated with infinite multiplicity or at least two points that are accumulation points in $K$, we can do a similar construction and have at least three nonisomorphic models of the theory of density character $\aleph_1$, even allowing for perturbations.
\end{rem}







    

\section{Stability}\label{sec:stability}

Throughout this section, we fix a completion $\Sigma$ of $\Sigma_{\operatorname{normal}}$ and a sufficiently saturated model $(H,T)$ of $\Sigma$.  Our first goal is to give a functional analytic characterization of nonforking independence.

\begin{defn}\label{defn:indeprelation}
For algebraically closed $B\subseteq
 C\subseteq H$ and a tuple $\vec a=(a_1,\dots,a_n)\in H^n$, we write $\vec a\ind_{B}C$ if $P_{B}(a_i)=P_C(a_i)$ for each $i=1,\ldots,n$. For general $A,B,C\subseteq H$, we write $A\ind_{B}C$ if $a \ind_{\langle B \rangle}\langle B\cup C\rangle$ for every $a\in A$.
\end{defn}

\begin{theo}\label{theo:char-forking}
 $\Sigma$ is stable and $\ind$ agrees with non-forking independence.   \end{theo}

\begin{proof}
We verify the usual list of axioms that guarantee a theory is stable and that the abstract independence relation characterizes nonforking independence.  Clearly, the notion of independence satisfies transitivity and finite character.

Stationarity: Consider tuples $\vec a=(a_1,\dots,a_n),\vec a'=(a_1',\dots,a_n')\in H^n$ and sets $B\subseteq C\subseteq H$ with $B$ algebraically closed.  Furthermore suppose that $\tp(\vec a/B)=\tp(\vec a'/B)$ and that $\vec a\ind_BC$, $\vec a'\ind_BC$.  For each $i=1,\ldots, n$, write $a_i=P_{B}(a_i)+c_i$ and
$a_i'=P_{B}(a_i)+c_i'$ with $c_i,c_i'\perp B$. Note that $\tp(c_1,\dots,c_n/\emptyset)=\tp(c_1',\dots,c_n'/\emptyset)$. Since $\vec a\ind_BC$ and $\vec a'\ind_BC$, we have $c_i,c_i'\perp \langle B\cup C\rangle$.  It follows that $\tp(\vec a/B\cup C)=\tp(\vec a'/B\cup C)$.

Symmetry: Consider $\vec a=(a_1,\dots,a_n)\in H^n$ and $\vec c=(c_1,\dots,c_k)\in H^k$ and assume that $\vec a\ind_B \vec c$. For any $j=1,\ldots, n$, any polynomials $p_1(T,T^*),\dots, p_k(T,T^*)$, and any $b\in \langle B\rangle$, we have $$a_j-P_{\langle B\rangle}(a_j)\perp p_1(T,T^*)c_1+\dots+p_k(T,T^*)c_{k}+b.$$ Since the family of polynomials is closed under applying $T$ and $T^*$ and $B$ is also invariant under $T$ and $T^*$, we also have, for 
any $s\geq 1$ and any $b\in B$, that
$$T^s(a_j)-P_{\langle B\rangle}(T^{s}a_j)\perp p_1(T,T^*)c_1+\dots+p_k(T,T^*)c_{k}+b,$$
$$(T^*)^s(a_j)-P_{\langle B\rangle}((T^*{s})a_j)\perp p_1(T,T^*)c_1+\dots+p_k(T,T^*)c_{k}+b.$$
Thus for any polynomial $t_j(T,T^*)$ and any $b\in B$, we also have $$t_j(T,T^*)a_j-P_{\langle B\rangle}(t_j(T,T^*)a_j)\perp p_1(T,T^*)c_1+\dots+p_\ell(T,T^*)c_{k}+b$$ and thus $$\sum_{j=1}^n(t_j(T,T^*)a_j-P_{\langle B\rangle}(t_j(T,T^*)a_j))\perp p_1(T,T^*)c_1+\dots+p_\ell(T,T^*)c_{\ell}+b.$$ In particular, for any $i=1,\ldots,k$ and any $b\in B$ we get 
$$c_i,P_{\langle B\rangle}(c_i),b \perp \sum _j (t_j(T,T^*)a_j-P_{\langle B\rangle}(t_j(T,T^*)a_j))$$ and thus $c_i-P_{\langle B\rangle}(c_i) \perp \sum _j t_j(T,T^*)a_j+b$. This shows that for any $i=1,\ldots,k$ we have $P_{\langle B \cup\{a_j \ : \ j=1,\ldots, n\}\rangle }(c_i)=P_{\langle B \rangle}(c_i)$ and thus $\vec c\ind_B \vec a$.

Existence: Let $B\subseteq C \subseteq H$ and let $\vec a\in H^n$. We will show that $\tp(\vec a/B)$ has a free extension to $C$ in an elementary extension of $(H,T)$ (and thus in $(H,T)$ by saturation).
Without loss of generality, we may assume that $B$ is algebraically closed. Set $D:=B^\perp$, so $D$ is a closed subspace of $H$ invariant under $T$, $T^*$ and it is orthogonal to $H_{\fin,T}$.
Let $(D',T')$ be a copy of $(D,T\upharpoonright_{D})$.
Then $(H,T)\oplus (D',T')$ is an elementary extension of $(H,T)$.  Write $a_i=b_i+d_i$, where $b_i=P_{B}(a_i)$ and $d_i=P_{D}(a_i)$. Let $d_i'$ be the copy of $d_i$ in the space $D'$
and let $a_i'=b_i+d_i'$. Then $\tp(a_1',\dots,a_n'/B)=\tp(a_1,\dots,a_n/B)$ and for each $i=1,\ldots, n$, we have $P_{H}(a_i')=P_{B}(a_i')$.  In particular, $P_{\langle C\rangle }(a_i')=P_{B}(a_i')$ for each $i=1,\ldots,n$, and so $\vec a'\ind_B C$, as desired.

Local character: Fix $B\subseteq H$ and $\vec a=(a_1,\ldots,a_n)\in H^n$. If we set $D=\{P_{\langle B\rangle}(a_1),\dots,P_{\langle B\rangle}(a_n)\}$, then by the definition of independence, we have $\vec a\ind_{D} B$. Since $D\subseteq \langle B\rangle$ and $D$ is finite, there is $B_0\subseteq B$ countable with $D\subseteq \langle B_0\rangle $. It follows that $\vec a\ind_{B_0} B$.
\end{proof}


\begin{theo}\label{prop:superstable}
$\Sigma$ is superstable.
\end{theo}

\begin{proof}
Fix $B_0\subseteq H$ and set $B=\langle B_0 \rangle$. 
By the definition of independence, it suffices to prove that for any single element $a\in H$ and any $\epsilon>0$, there are a finite $B_\epsilon\subseteq B_0$ and $a'\in H$ such that $d(a,a')<\epsilon$ and $a'\ind_{B_\epsilon}B$.
Write $a=P_B(a)+a_1$, where $a_1\in B^\perp$. By Theorem \ref{theo:char-forking}, we have $a\ind_{P_B(a)} B$.

Set $h=P_{H_{\fin,T}}(a)$. Since $P_{B}(x)\in B$, given $\epsilon>0$, we can find $n\geq 0$ and a sum $p_1(T,T^*)(b_1)+\dots+p_n(T,T^*)(b_n)$, where $b_1,\dots,b_n\in B_0$, such that $$\|P_B(x)-(h+p_1(T,T^*)(b_1)+\dots+p_n(T,T^*)(b_n))\|<\epsilon.$$ Set $B_\epsilon=\{a_1,\dots,a_n\}$ and $a'=(h+p_1(T,T^*)(b_1)+\dots+p_n(T,T^*)(b_n))+a_1$. Then $d(a,a')=d(P_{B}(a),P_{B}(a'))<\epsilon$ and $a' \ind_{B_{\epsilon}} B$.
\end{proof}

\begin{prop}
$\Sigma$ is not multidimensional.
\end{prop}

\begin{proof}
    Fix $B\subseteq H$ and $\vec a=(a_1,\dots,a_n)\in H^n$. Without loss of generality, we may assume that $B=\langle B\rangle$. For each $i=1,\ldots, n$ we can write $a_i=a_{iB}+c_i$, where $a_{iB}:=P_{B}(a_i)$. Set $\vec c=(c_1,\dots,c_n)$ and note that by the characterization of independence, $\tp(c_1,\dots,c_n/B)$ is a non-forking extension of $\tp(c_1,\dots,c_n/\emptyset)$ and that $\vec a\nind_{B} \vec c$. Consequently, $\tp(\vec a/B)$ is not weakly orthogonal to $\tp(\vec c/B)$ and thus $\tp(\vec a/B)$ is not orthogonal to $\tp(\vec c/\emptyset)$.
\end{proof}

\begin{prop}
 Let $B\subseteq H$ be algebraically closed and let $\vec a=(a_1,\dots,a_n)\in H^n$. Then $\operatorname{Cb}(\vec a/B)$ is interdefinable with the set of elements $C=\{P_B(a_i)\}_{i\leq n}$.
\end{prop}
\begin{proof}
    By Theorem \ref{theo:char-forking}, we have that $\vec a \ind_C B$, and since $\tp(\vec a/C)$ is stationary (see the proof of stationarity in Theorem \ref{theo:char-forking}, noting that the same proof works in the current setting even though $C$ is not algebraically closed), we obtain $\operatorname{Cb}(\vec a/B)=\operatorname{Cb}(\vec a/C)\subseteq \dcl^{\operatorname{eq}}(C)$.

    For the converse, it suffices to show that, for every $i=1,\ldots,n$, the projection $P_{B}(a_i)$ belongs to the definable closure of any Morley sequence in $\operatorname{tp}(\vec a/B)$. Towards this end, let $(\vec a^k)_{k\in\mathbb N}=(a_1^k,\dots,a_n^k)_{k\in\mathbb N}$ be a Morley sequence with $\vec a^0=\vec a$. Then $P_{B}(a_i^k)=P_{B}(a_i^0)$ for all $i\leq n$, so $\{a_i^k-P_{B}(a_i^0)\}_{n\in\mathbb N}$ forms an orthogonal sequence of bounded norm and converges, in the sense of Ces\`aro, to $0$. Thus, for any $i=1,\ldots,n$, we have
     $$\sum_{k=0}^{m-1}\dfrac{a_i^k}{m}=P_{B}(a_i^0)+\sum_{k=0}^{m-1}\dfrac{a_i^k-P_{B}(a_i^k)}{m}\to P_{B}(a_i^0),$$
     yielding the desired result.
\end{proof}

\begin{cor}
The theory $\Sigma$ has weak elimination of imaginaries.  
\end{cor}

In what follows, the cardinality of $K$ will be relevant, so we go back to our standard notion and talk again about theories of the form $\Sigma_{K,m}$ instead of just $\Sigma$.

\begin{theo}\label{prop:omega-stability}
$\Sigma_{K,m}$ is $\omega$-stable if and only if $K$ is countable.
\end{theo}

\begin{proof}
Assume first that $K$ is uncountable. Since $K$ is a closed uncountable subset of $\mathbb{C}$, it contains a nonempty perfect set $P$. For every $\lambda\in P$, consider the type $p_{\lambda}(x)$ over the emptyset determined by $\|x\|=1$ and $Tx=\lambda x$. Since $\lambda$ is not an isolated point of $K$, $p_\lambda$ is a non-algebraic type.  In addition, for distinct $\lambda,\lambda'\in P$, given realizations $a\models p_{\lambda}$ and $b\models p_{\lambda'}$, we have that $a$ and $b$ belong to distinct eigenspaces and thus $d(a,b)=\sqrt{2}$. The family 
of types $\{p_{\lambda}:\lambda\in P\}$ is uniformly separated, each one is a type over $\emptyset$, and there are $2^{\aleph_0}$ many such types, whence $\Sigma_{K,m}$ is not $\omega$-stable.

Now assume that $K$ is countable. Enumerate $K=\{\lambda_i: i\in I\}$ with $|I|\leq \aleph_0$. To simplify the notation, set $H_i$ to be the eigenspace associated to $\lambda_i$ and write $P_{H_i}$ for the orthogonal projection onto this space. By Lemma \ref{lem:countable}, we have that $H=\bigoplus_{i\in I} H_i$, so for each $a\in H$, we can write $a=\sum_{i\in I}P_{H_i}(a)$. For each $i\in I$, we say that $\lambda_i$ belongs to the \emph{support} of $a$ if $P_{H_i}(a)\neq 0$.
For $\epsilon>0$, choose a finite $I_{a,\epsilon}\subset I$ such that $$\left\|\sum_{i\in I}P_{H_i}(a)-\sum_{i\in I_{a,\epsilon}}P_{H_i}(a)\right\|<\epsilon.$$
Note that the support of $a_\epsilon:=\sum_{i\in I_{a,\epsilon}}P_{H_i}(a)$ is contained in the finite set $I_{a,\epsilon}$. 
This shows that elements with finite support are dense in $H$, whence the collection of types over $\emptyset$ of elements of the unit ball with finite support form a separable subspace of the space of types over $\emptyset$ of elements of the unit ball. 

Now take $B\subseteq H$ countable, whence the subspace $\langle B\rangle$
is separable. For any $a,c\in H$, we can write $a=a_B+a_1$, $c=c_B+c_1$, with $a_B=P_{\langle B\rangle}(a)$, $c_B=P_{\langle B\rangle}(c)$, 
 and $a_1,c_1\perp \langle B\rangle$ and thus $\|a-c\|^2=\|a_B-c_B\|^2+\|a_1-c_1\|^2$. In order to show $\Sigma_{K,m}$ is $\omega$-stable, it suffices to prove that:  (i) there is a separable family of types over $\emptyset$ of elements inside $\langle B\rangle$, and (ii) types of elements over $\emptyset$ orthogonal to $\langle B\rangle$. Since $\langle B\rangle$ is separable, (i) is clear. For (ii), use a dense countable collection of types of elements of finite support that belong to $\langle B\rangle^\perp$, which is possible by the previous paragraph (after passing to $T\upharpoonright \! \langle B\rangle^\perp)$.
\end{proof}

\begin{rem}
Assume $K$ is countable.  Since $\Sigma_{K,m}$ is $\omega$-stable,  it follows that $\Sigma_{K,m}$ has a separable saturated model and a separable prime model. These can be concretely described in our setting.  For each $\lambda \in \isol(K)$, the space $H_\lambda$ is definable and its dimension is the same in all models of the theory, namely $m(\lambda)$. 
On the other hand, for each accumulation point $\lambda \in K'$ and any value $n_{\lambda}\in \{0,1,\dots,\aleph_0\}$, we can find a model $(H,T)\models \Sigma_{K,m}$ with $\dim(H_{\lambda})=n_{\lambda}$. The saturated separable model of $\Sigma_{K,m}$ is the one where $\dim(H_{\lambda})=\aleph_0$ for all $\lambda\in K'$ while the prime model is the one with $\dim(H_{\lambda})=0$ for all $\lambda\in K'$.
\end{rem}

We now discuss $\omega$-stability and existence of prime models up to perturbation.  To do so, we need to be precise about what these mean in our context.

\begin{defn}\label{def:pert}
    For $i=0,1$, fix models $(H_i,T_i)\models\Sigma_{K,m}$ for $i=0,1$.  Fix also $r\geq 0$. We define an \textbf{$r$-perturbation of} from $(H_0,T_0)$ to $(H_1,T_1)$ to be an isometric isomorphism of Hilbert spaces $U:H_0\to H_1$ for which
    $$\| UT_0U^{-1}-T_1\|\leq r.$$
    The set of all $r$-perturbations from $(H_0,T_0)$ to $(H_1,T_1)$ will be denoted $$\operatorname{Pert}_r((H_0,T_0),(H_1,T_1))$$ and simply by $\operatorname{Pert}_r(H_0,T_0)$  if $(H_0,T_0)=(H_1,T_1)$. 
    Assume now that $(H_0,T_0)$ is sufficiently saturated and strongly homogeneous.  Fix $p,q\in S_n(T)$. Given $r\geq 0$, we say that $d_{\pert}(p,q)\leq r$ if there are realizations $\vec a\models p$, $ \vec b\models q$ and $U\in \operatorname{Pert}_r(H_0,T_0)$ for which $\|U(\vec a)-\vec b\|\leq r$.
    We say $\Sigma_{K,m}$ is \textbf{$\omega$-stable up to perturbations} if for any separable $(H_0,T_0)\models \Sigma_{K,m}$, 
 $(S_1(H_0),d_{\pert})$ is also separable. 
\end{defn}

\begin{theo} Every theory
$\Sigma_{K,m}$ is $\omega$-stable up to perturbations.
\end{theo}

\begin{proof}
If $K$ is countable, the result follows from Proposition \ref{prop:omega-stability}. Thus, without loss of generality, it is enough to consider the case where $|K|=2^{\aleph_0}$. 

Let $(H_0,T_0)\models \Sigma_{K,m}$ be separable. We will show all types over $H_0$ can be realized approximately in a separable elementary extension of $(H_0,T_0)$. Let $(H_1,T_1)\models \Sigma_{\operatorname{normal}}$ be a separable structure with $\sigma(T_1)=K'$ such that all eigenspaces associated to isolated points in $K'$ have dimension $\aleph_0$.  (In other words, $(H_1,T_1)$ is a separable model of the model companion of $\Sigma_{K'}$; see Corollary \ref{modcomp}.) Set $(H_2,T_2)=(H_0,T_0)\oplus (H_1,T_1)$. Then $\sigma(T_2)=\Sigma$ and the isolated points in $K$ belong to $\sigma(T_0)\setminus \sigma(T_1)$, so their multiplicity does not increase when we go from the structure $(H_0,T_0)$ to the structure $(H_2,T_2)$. This shows that $(H_0,T_0)$ and $(H_2,T_2)$
are spectrally equivalent and thus, by Corollary \ref{Cor:spec-elem}, we have $(H_2,T_2)\models \Sigma_{K,m}$.  By model completeness of $\Sigma_{K,m}$, $(H_2,T_2)$ is an elementary extension of $(H_0,T_0)$.  Notice also that $(H_2,T_2)$ is a ``rich'' elementary extension of $(H_0,T_0)$ in the sense that all eigenspaces in $(H_2,T_2)$ associated to isolated points of $K'$ have dimension $\aleph_0$. 

Now take a 1-type $p$ over $H_0$; we claim that $p$ is approximately realized in $(H_2,T_2)$.  Take a separable elementary extension $(H_3,T_3)$ of $(H_0,T_0)$ where $p$ is realized.  Set $(H_4,T_4)= (H_3,T_3)\oplus (H_1,T_1)$. As above, $(H_4,T_4)$ is a ``rich'' separable elementary extension of $(H_0,T_0)$ containing a realization of $p$. Set $H_5:=H_4\ominus H_0$ and $T_5:=T_4\upharpoonright\!H_5$.  Then $\sigma(T_5)=K'$ and all isolated points of $K'$ have associated  eigenspaces of
dimension $\aleph_0$. The structures $(H_5,T_5)$ and $(H_1,T_1)$ are both separable, $\sigma(T_1)=\sigma(T_5)=K'$, 
and in both models all isolated points of $K'$ have eigenspaces of dimension $\aleph_0$. In other words,  $(H_5,T_5)$, $(H_1,T_1)$ are spectrally equivalent and thus there is a sequence $(U_n)_n$ of unitary maps $U_n:H_1\to H_5$ witnessing that $\lim_{n\to \infty} \|T_1-U_n^{-1}T_5U_n \|=0$.
Set $V_n=\operatorname{Id}\oplus U_n$, where $\operatorname{Id}:H_0\to H_0$ is the identity map. Then $V_n:H_2\to H_4$ witnesses that $\lim_{n\to \infty} \|T_2-V_n^{-1}T_4V_n \|=0$
and $V_n(a)=a$ for all $a\in H_0$.  If $b\in H_4$ realizes $p$, write $b=P_{H_0}(b)+c$ with $c\in H_4\ominus H_0=H_5$.  It follows that, for $n$ large enough, $P_{H_0}(b)+V_n^*(c)\in H_2$ almost realizes $p$, as desired.
\end{proof}

\begin{defn}
    Given 
$(H_0,T_0)\models \Sigma_{K,m}$ and $A\subseteq H_0$, we write $L_A$ for the language obtained by adding a constant for each element in $A$ and $\Sigma_{K,m,A}$ for the $L_A$-theory obtained by adding sentences for the formula that elements of $A$ satisfy in $(H_0,T_0)$.   
We say that $\Sigma_{K,m}$ \textbf{has prime models up to perturbations} if given a model $(H_0,T_0)$ of $\Sigma_{K,m}$ and $A\subseteq H_0$, there is  a model $(H_1,T_1;c_a\in A)$ of $ \Sigma_{K,m,A}$ such that for all models
    $(H_2,T_2;c_a;a\in A)$ of $\Sigma_{K,m,A}$ and all $r\geq 0$, there is an elementary substructure $(H_1',T_1';c_a:a\in A)\preceq (H_2,T_2;c_a:a\in A)$ with $A\subseteq H_1'$ and $U_r\in \text{Pert}_r((H_1,T_1),(H_1',T_1'))$ that fixes $A$ pointwise in the sense that $U_r(c_a^{H_1})=c_a^{H_2}$ for each $a\in A$.
\end{defn}

A similar argument as above shows the following:

\begin{theo} Every theory
$\Sigma_{K,m}^*$ has prime models up to perturbations.
\end{theo}

\begin{proof}
    Just as we used ''rich'' extension to show that the theory $\Sigma_{K,m}$ is $\omega$-stable up to perturbations, we can use ''poor'' models to show that the theory $\Sigma_{K,m}^*$ has prime models up to perturbations over any set $B$. By first replacing $B$ with $\acl(B)$, it is sufficient to construct a prime model over a substructure $(H_0,T_0,T_0^*)$ of a model of $\Sigma_{K,m}^*$ that contains $(H_{\fin},T_{\fin},T_{\fin}^*)$. Note that $\sigma(T_0)\subseteq K$ and that $H_0$ includes all finite dimensional eigenspaces associated to isolated points in the spectrum.
   If there are isolated points $\lambda$ of $K$ with $m(\lambda)=\infty$ but $H_{0,\lambda}$ is finite-dimensional, when constructing the prime extension, simply add countable many orthonormal eigenvectors to each of these eigenspaces. As a result, we may assume all eigenspaces corresponding to isolated points have the correct multiplicity in $(H_0,T_0,T_0^*)$.
   
  Write $K=P\cup C$, where $P$ is a perfect subset of $K$ and $C$ is countable. The set $C$ is the collection of isolated points in $K$ together with some of its accumulation points; as a result, $C\subseteq \sigma(T_0)$. (We will not add more eigenvectors associated to the elements of $C$, so in that sense our extension will be ``poor''.)  If $\sigma(T_0)\cap P= P$, then $\sigma(T_0)=P\cup C$ and for all isolated points $\lambda$ of $K$, we have $m(\lambda)=\dim(H_{0,\lambda})$.  As a result,
$(H_0,T_0,T_0^*)\models \Sigma_{K,m}^*$ and we are done. Thus we may assume that
$\sigma(T_0)\cap P\subsetneq P$ and we need to add witnesses for the part of the spectrum
    that we are missing. Set $S=P\setminus \sigma(T_0)$, which is an open subset of $K$ disjoint from all isolated points in $K$.
    Take separable $(H_2,T_2,T_2^*)\models \Sigma_{\overline{S}}^*$ and set $(H_1,T_1,T_1^*)=(H_0,T_0,T_0^*)\oplus (H_2,T_2,T_2^*)$.  Note that $\sigma(T_1)=K$ and that we did not modify the eigenspaces associated to the isolated point from $(H_0,T_0,T_0^*)$.

We claim that $(H_1,T_1,T_1^*)$ is prime over $(H_0,T_0,T_0^*)$ up to perturbations. Given another extension $(H_3,T_3,T_3^*)\supseteq (H_0,T_0,T_0^*)$ with $(H_3,T_3,T_3^*)\models \Sigma_{K,m}^*$ we can write
$(H_3,T_3,T_3^*)= (H_0,T_0,T_0^*)\oplus (H_4,T_4,T_4^*)$ with $\sigma(T_4)\supseteq \overline{S}$. As $\sigma(T_2)=\overline{S}$ and $S$ is open, we have $\sigma(T_2 \oplus T_4)=\sigma(T_4)$ and $(H_4,T_4,T_4^*)\oplus (H_2,T_2,T_2^*)$ does not modify the dimensions of the eigenspaces associated to isolated points of the spectrum $\sigma(T_4)$.
This shows that $(H_4,T_4,T_4^*)\oplus (H_2,T_2,T_2^*)$ is spectrally equivalent to $(H_4,T_4,T_4^*)$ and thus $T_2$ is approximately unitarily equivalent to $T_4\oplus T_2$.  As a result, we can embed, up to perturbations, the space $(H_2,T_2,T_2^*)$ into $(H_4,T_4,T_4^*)$, completing the proof.
\end{proof}

From the point of view of classification theory, the theories $\Sigma_{K,m}$ are very well behaved, even without considering perturbations. A futher example of this is a result of the second and third named authors of this paper, together with N. Levi, showing that the theories $\Sigma_{K,m}$ have the Schr\"oder Bernstein property \cite{SB}. The stability-theoretic properties of the theories $\Sigma_{K,m}$ could be developed even further. For instance, in \cite[Section 4]{BYUsZa}, in the special case of a unitary operator with full spectrum, the authors characterize orthogonality and domination of types in terms of properties of the corresponding measures (for example, orthogonality of types over $\emptyset$ agrees with orthogonality of the associated measures).

\section{Open Questions}

We end the paper with a couple of further questions that stem naturally from the results presented in this paper.

First, we characterized the logic topology in the corresponding space of measures as the weak* topology though the maps $\Phi_n$ studied in Section \ref{types-categ}. This leads to the following questions, the first of which was also discussed in Section \ref{types-categ}:

\begin{ques}
Fix a completion $\Sigma=\Sigma_{K,m}$ of $\Sigma_{\operatorname{normal}}$.

\begin{enumerate}
 \item Consider the metric $d^{\Phi}$ on $M(K)_+$ induced by the metric on $S_1(\Sigma)$ via the identification given by $\Phi:S_1(\Sigma)\to M(K)_+$. Does the topology induced by $d^{\Phi}$ correspond to a natural topology on $M(K)_+$?

 \item Consider the metric $d^{\Phi}_{\pert}$ on $M(K)_+$ induced by the metric $d_{\pert}$ on $S_1(\Sigma)$ via the identification given by $\Phi:S_1(\Sigma)\to M(K)_+$. Does the topology induced by $d^{\Phi}_{\pert}$ correspond to a natural topology on $M(K)_+$?

 \end{enumerate}

\end{ques}

Next, given a theory
$\Sigma=\Sigma_{K,m}$, one can build the theory of beautiful (belle) pairs $\Sigma_P$ obtained by adding a predicate $P$ corresponding to the distance to a Hilbert subspace which is invariant under $T$ and $T^*$. We can add axioms saying that the underlying Hilbert space and the substructure both model the theory $\Sigma$. Finally, for each point $ \lambda\in K$ with $m(\lambda)=\infty$, we add the scheme

$$\inf_{x_1}\dots \inf_{x_n}\sup_{y:P(y)=0}\max\left(\max_{i,j\leq n}\|\langle x_i,x_j\rangle-\delta_{i,j}\|,\max_{i\leq n}\|T(x_i)-\lambda x_i\|,\max_{i\leq n}\|\langle x_i,y\rangle\| \right)$$
saying that the approximate codimension of the eigenspaces associated to $\lambda$ are infinite.

\begin{ques}
Characterize forking in $\Sigma_P$ and determine if some traces of the uniform finite basedness of Hilbert spaces gets reflected in properties of $\Sigma_P$.
\end{ques}


\end{document}